\documentclass[3p]{amsart}
\usepackage{amsmath,amssymb,amsbsy,amsthm,mathrsfs}

\usepackage[margin=1.05in]{geometry}

\usepackage{algorithm}
\usepackage{algpseudocode}
\usepackage[dvipsnames]{xcolor}
\usepackage{tikz}
\usetikzlibrary{decorations.pathmorphing}
\usepackage{pgfplots}
\usepackage{lmodern}
\usepackage{booktabs}
\usepackage{multirow}
\usepackage{thmtools}
\usepackage{hyperref}
\hypersetup{colorlinks = true, allcolors = blue}
\usepackage[nameinlink,noabbrev]{cleveref}

\newtheorem{thm}{Theorem}
\newtheorem{cor}[thm]{Corollary}

\newtheorem{prop}[thm]{Proposition}
\newtheorem{rem}[thm]{\em Remark}
\newtheorem{ex}[thm]{\em Example}

\pgfplotsset{compat=1.18}

\begin{document}

\title{Convergence Analysis of the Random Bisection Method} 

\author{Ludovick Bouthat}

\author{Philippe-André Luneau}

\author{Philippe Petitclerc}

\address{D\'epartment de Math\'ematiques et Statistiques, Universit\'e Laval, Qu\'ebec, Canada}

\begin{abstract}
We propose a generalized version of the bisection method where the cutting point between the two subintervals is chosen at random following an arbitrary distribution. We compute expected convergence rates with respect to any arbitrary absolutely continuous a priori distribution for the position of the root in the initial interval and prove that it depends only on the expectation $\mathbb{E}[c(1-c)]$ of the cut $c$. We also provide a generalization of the method for $K$ random cuts and study its convergence properties. Our theoretical results are then validated numerically using statistical simulation.
\end{abstract}



\setlength{\skip\footins}{1.2pc plus 5pt minus 2pt}

\maketitle

\vspace{-4pt}
\section{Introduction}

\subsection{Background and motivation}

Root-finding algorithms have been part of the mathematician's toolbox for centuries (their usage can be traced back as far as the first century CE, by Hero of Alexandria \cite{Earl_Nicholson_2021}). In general, it is not possible to find analytically the solutions to a nonlinear equation. Those equations arise naturally from problems in natural sciences or engineering, and must often be solved in the context of numerical simulation of physical or industrial processes. For example, the well-known Newton's method and its variants are an essential part of modern scientific computing and optimization software. However, to achieve a sufficient level of efficiency and robustness, this class of algorithms generally requires first or even second-order information. In many contexts, such as blackbox optimization, this information is unknown, or at best, approximated (e.g., Broyden-class algorithms \cite{Byrd_Liu_Nocedal_1992}). This naturally motivates the use of derivative-free algorithms for root-finding and optimization.

There are two large families dominating the field of derivative-free algorithms. The oldest, bracketing methods, have been around for a very long time because of their simplicity and reliability, as they ensure global convergence under weak hypotheses \cite{Suhadolnik_2012}. The principle of such methods is simple: the algorithm builds a sequence of nested intervals containing the given root or minimizer. The lengths of the intervals decrease at each iteration, such that after a large number of iterations, the root or minimizer can be precisely approximated. Well-known examples of this approach are the \emph{regula falsi} method (false position), the golden-section search for optimization, and the bisection method, the latter being the central theme of this article. 

A more recent subclass of derivative-free methods is the family of probabilistic search algorithms. These techniques, based on function evaluations, can assign a probability of finding a root or minimizer to regions of the search space (for instance by using a probabilistic data-driven model), and then choose the next iterate according to a deterministic or nondeterministic rule. With the increasing quantity of data extracted from natural and industrial processes and the rise in popularity of machine learning, these techniques have become more prevalent over the years. Examples of algorithms following this paradigm include Bayesian optimization \cite{garnett_bayesoptbook_2023} and metaheuristics, like simulated annealing or evolutionary algorithms  \cite{Gendreau_Potvin_2019,Li_Gong_Lim_Liao_Gu_2024}.

The algorithm studied here lies precisely at the intersection of these two classes. In the classical bisection algorithm, the nested intervals are generated by simply cutting the previous interval in half, and keeping the subinterval containing the root. In this work, we investigate what happens when this elementary deterministic procedure is randomly perturbed. More precisely, instead of cutting the interval at its midpoint, the cutting point is chosen at random according to a given distribution. This seemingly minor modification transforms the algorithm into a stochastic dynamical system whose probabilistic behavior turns out to have surprisingly regular structure.

At first glance, one might expect that if the cut is chosen from a distribution centered at the midpoint, then the method should retain, at least on average, some of the behavior of the classical bisection algorithm. Numerical experimentation quickly showed that this intuition was misleading. Indeed, the randomized procedure exhibited a slower convergence rate and several regular probabilistic patterns that suggested the existence of invariant distributions and explicit convergence laws. This paper naturally grew out of the effort to understand these observations.

Using elementary probabilistic and geometric arguments, we derive explicit formulas for the distribution of the interval reduction factor and its first moments. We show that the uniform distribution is stationary for the normalized root process, \emph{independently of the distribution of the cuts}, and establish convergence to this invariant distribution for any initial absolutely continuous root distribution, provided that the cut distribution gives positive mass to the interior of the interval. Consequently, we obtain explicit expressions for the expected convergence rate of the stochastic bisection method, depending only on $\mathbb{E}[c(1-c)]$. These results yield a probabilistic optimality principle: among all such randomized variants, the classical deterministic bisection method is essentially the only one achieving the smallest expected contraction factor.

We also study a natural generalization in which $K>1$ random cuts are chosen at each iteration. Although increasing the number of cuts is expected to accelerate convergence, we show that the observed acceleration can behave in a counter-intuitive way for random multisection strategies. In particular, the uniform random trisection procedure has, on average, the same contraction factor as the classical deterministic bisection method.

\subsection{The Stochastic Bisection Algorithm and Related Definitions}

Given a continuous function  $f:\mathbb{R}\to\mathbb{R}$ with a root $r$, i.e., $f(r)=0$, the bisection method can approximate the root's value if it is known to lie in an interval $[a,b]$ such that $f$ has opposite signs at the endpoints. The cut function is defined as the midpoint of the interval:
\[
\mathrm{cut}(a,b) = \frac{a+b}{2}.
\]

It is easily shown that the interval generated by this algorithm has length $\bigl|[a_n,b_n]\bigr| = 2^{-n}(b-a)$. If the cuts $c_n$ are viewed as approximations of the root $r$, the error $e_n = |c_n-r|$ is always bounded by the interval half-length, i.e., $e_n \leq 2^{-(n+1)}(b-a)$. Thus, the algorithm converges linearly with rate $\frac{1}{2}$, that is, $e_{n+1}\approx \frac{1}{2}e_n$. The algorithm generally applies to functions with multiple roots in $[a,b]$, but here we assume the root is unique.

\vspace{-6pt}
\begin{algorithm}[ht]
\caption{Bisection Method}\label{alg:bisection}
\begin{algorithmic}
\Require $a,b$ s.t.~$a<b$, $f$ s.t.~$f(a)f(b)<0$ with root $r\in[a,b]$, $\varepsilon>0$, $N>0$
\Ensure $r\in [a_n,b_n]$
\State $n \gets 0$
\State $a_n \gets a$
\State $b_n \gets b$
\While{$b_n-a_n \geq \varepsilon$ and $n < N$}
\State $c_n \gets \mathrm{cut}(a_n,b_n)$
\If{$f(a_n)f(c_n)<0$}
    \State $a_{n+1} \gets a_n$
    \State $b_{n+1} \gets c_n$
\Else
    \State $a_{n+1} \gets c_n$
    \State $b_{n+1} \gets b_n$
\EndIf
\State $n \gets n+1$
\EndWhile\\
\Return $[a_n,b_n]$
\end{algorithmic}
\end{algorithm}

Because of its simplicity and robustness, the bisection method has inspired a number of generalizations. For instance, one natural extension of the bisection algorithm is to increase the number of partitions of the current interval: the \emph{multisection methods} (or $K$-section), of which bisection and trisection are particular cases, divide the interval into $K$ equal subintervals \cite{Badr_Almotairi_Ghamry_2021}. These methods have faster theoretical convergence rates than classical bisection, at the cost of more function evaluations per iteration. Higher-dimensional bracketing methods inspired by bisection, for instance based on simplices, have also been introduced \cite{Kearfott_1987,Wood_1992}.

Early work by Zemel \cite{Zemel_1986} and Kennedy \cite{Kennedy_1988} studied a very similar problem, minimization with random sampling, using respectively a bisection algorithm on $f'$ and the Golden Section method. The latter derived the interesting result that the limit point of a randomized Golden Section algorithm follows a Beta distribution. The present article lies in the direct continuation of these works and establishes analogous results for the random bisection algorithm.

Another direction has been to adapt the bisection algorithm to settings where the function is uncertain or noisy. The Probabilistic Bisection Algorithm (PBA) \cite{Frazier_Henderson_Waeber_2019,Horstein_1963,Rodriguez_Ludkovski_2020,Waeber_Frazier_Henderson_2013} constructs and updates a stochastic model for the root position based on function queries. Waeber, Frazier, and Henderson \cite{Waeber_Frazier_Henderson_2013} proved, using random-walk theory, that the expected error of PBA converges linearly to $0$, as in the classical bisection method. Other stochastic variants involving the function $f$ have also been studied; for example, \cite{Emiris_Galligo_Tsigaridas_2010,Kumari_Chakraborty_2024} analyze the expected behavior and complexity of the bisection method for random polynomials.

The stochastic variant considered in the present paper follows a different philosophy. It is formally analogous to the classical bisection method, but with one crucial modification: the function $f$ remains deterministic, while the cut $c_n$, instead of being chosen as the midpoint of the current interval, is sampled at random from a distribution $\mathcal{D}$ supported on $[a_n,b_n]$:
\begin{equation*}
    \mathrm{cut}(a_n,b_n) = c_n \sim \mathcal{D}(a_n,b_n), \qquad n\in\mathbb{N}.
\end{equation*}
In this sense, our framework differs fundamentally from the PBA: randomness stems not from noisy evaluations or function uncertainty, but from the algorithmic choice of the cut itself.

Our main objective is to determine the \emph{expected convergence rate} of such a method. To simplify the analysis of the cuts and endpoints, we instead analyze a scale-invariant algorithm (\Cref{alg:RRbisection}) generating the interval length. The starting interval is taken to be $[0,1]$ without loss of generality, and the root is normalized to its relative position in $[0,1]$ at each iteration. Note that this algorithm is not meant for practical use, since $r$ is generally unknown.

\vspace{-1.5pt}
\begin{algorithm}[ht]
\caption{Rescaled Bisection Method}\label{alg:RRbisection}
\begin{algorithmic}
\Require A point $r\in[0,1]$, $\varepsilon>0$, $N>0$
\Ensure $\ell_n$ and $L_n$ are respectively the scaling factor and the length of the interval at iteration $n$
\State $n \gets 0$
\State $a_n \gets 0$
\State $b_n \gets 1$
\State $\ell_n \gets \infty$
\State $L_n \gets 1$
\State $r_n \gets r$
\While{$L_n \geq \varepsilon$ and $n<N$}
\State $c_n \gets \mathrm{cut}(0,1)$
\If{$c_n>r_n$}
    \State $a_{n+1} \gets 0$
    \State $b_{n+1} \gets c_n$
\Else
    \State $a_{n+1} \gets c_n$
    \State $b_{n+1} \gets 1$
\EndIf
\State $r_{n+1}\gets \frac{r_n - a_{n+1}}{b_{n+1}-a_{n+1}}$ \Comment{Rescaling of the root to $[0,1]$}
\State $\ell_{n+1} \gets (b_{n+1}-a_{n+1})$
\State $L_{n+1} \gets \ell_{n+1}L_n$

\State $n \gets n+1$
\EndWhile\\
\Return $\ell_n$, $L_n$
\end{algorithmic}
\end{algorithm}
\vspace{-1.5pt}

In order to state our results in a way that applies to all random variables on $[0,1]$, including those that do not admit a probability density function, we use the Lebesgue--Stieltjes integral. Recall that any real-valued random variable $X$ on $[0,1]$ is characterized by its cumulative distribution function (CDF) $F(x)=\mathbb{P}(X\leq x)$, which is nonnegative, nondecreasing, right-continuous, and satisfies $F(0)=\mathbb{P}[X=0]$ and $F(1)=1$. Even when $F$ is not differentiable, one can still define, for any bounded measurable function $g$, integrals of the form
\begin{align*}
    \mathbb E[g(X)] = \int_{0}^{1} g(x)\,\mathrm{d} F(x).
\end{align*}
When $F$ happens to be differentiable with density $f$, the above reduces to the familiar formula
\begin{align*}
    \int_{0}^{1} g(x)f(x)\,\mathrm{d} x.
\end{align*}
Thus the Lebesgue--Stieltjes integral provides a unified notation that simultaneously covers continuous, discrete, and mixed distributions, avoiding the need to treat these cases separately. In what follows, this framework allows us to formulate our arguments in a concise and general way, without imposing unnecessary regularity assumptions such as the existence of a density.

\subsection{Stochastic Dynamical Systems}

In \Cref{alg:RRbisection}, the rescaling transformation can be written explicitly as $r_{n+1} = T(c_n,r_n)$, where
\begin{equation}\label{eq - iteration_root}
    T(c,r) = 
    \begin{cases}
        \frac{r}{c}& ~\text{if }\,c\geq r,\\
        \frac{r-c}{1-c} & ~\text{if }\,c < r,
    \end{cases}
\end{equation}
which corresponds to the skewed Dyadic map \cite{Roslan_2018} with parameter $1/c$ (\Cref{fig:skewtentmap}). The iterated \emph{Dyadic map} ($c=1/2$) is a well-known chaotic dynamical system which is topologically mixing \cite{Rodriguez-Gonzalez_2023}, meaning that for any open sets $A,B\subset [0,1]$, $T^n(1/2,A)\cap B \neq \varnothing$ for all $n$ large enough. Intuitively, under suitable regularity assumptions, the mixing nature of this transformation will tend to homogenize non-uniform root distributions (this will be shown precisely in \Cref{sec:3}). 

\begin{figure}[ht]
    \centering
    \begin{tikzpicture}[scale=.85]

    \definecolor{bluefill}{RGB}{25,25,112}
  \begin{axis}[
    width=7.58cm, height=7.58cm,
    xmin=0, xmax=1,
    ymin=0, ymax=1,
    xlabel={$r$},
    ylabel={$T(c,r)$},
    ylabel style={yshift=-8pt},
    axis lines=left,
    xtick={0, 0.3, 1},
    xticklabels = {$0$, $c$, $1$},
    ytick={0, 1},
    grid=major,
    grid style={dashed, gray!40},
    x label style={at={(axis description cs:1.03,0)},anchor=west},
  ]

  \addplot[bluefill, thick, domain=0:0.3, samples=2] {x/0.3};

  \addplot[bluefill, thick, domain=0.3:1, samples=2] {(x-0.3)/(0.7)};

  \end{axis}
\end{tikzpicture}
\vspace{-6pt}
    \caption{Skewed Dyadic map with parameter $1/c$.}
    \label{fig:skewtentmap}
\end{figure}

Because the stochastic bisection algorithm generates an iterated random map on the unit interval, it is tempting to view it in the framework of stochastic dynamical systems and ergodic theory. In this setting, one studies stationary distributions, invariant measures, and convergence toward equilibrium for random compositions of transformations. A classical result in this field is due to Diaconis and Freedman \cite{Diaconis_Freedman_1999}, who gave sufficient conditions for the uniqueness of a stationary distribution for random iterated Lipschitz maps. This result requires that the map is contractive on average, which is not satisfied by the map $T$.
Indeed, the random skewed Dyadic map is expanding on average, as is the deterministic Dyadic map. Hence, contraction-based ergodic results such as the one of Diaconis and Freedman cannot be used directly to establish or identify a stationary distribution in our setting.

Expansive maps have long been studied in dynamical systems. As early as 1973, Lasota and Yorke \cite{Lasota_Yorke_2004} used the Frobenius--Perron operator to prove $L^1$ convergence of Cesàro averages for piecewise $\mathcal{C}^2$ expanding maps of the unit interval, thereby establishing the existence of one or infinitely many invariant distributions for this class. These ideas were later extended to random maps with finitely many branches \cite{Keller_1982,Pelikan_1984}, and related systems continue to be investigated in recent work \cite{Yan_Majumdar_Ruffo_Sato_Beck_Klages_2024}. Other interval maps with random parameters, such as skewed tent maps \cite{Biswas_Seth_2018} and $\beta$-transformations \cite{Suzuki_2024}, exhibit dynamical behaviors qualitatively similar to those of the transformation $T(c,r)$ considered here, including chaotic and mixing properties.

Nevertheless, the stochastic bisection algorithm differs from these models in a significant way: the randomness enters through a continuous parameter governing the slope of the map at each iteration, rather than through a finite set of branches with stationary probabilities.
As a result, these ergodic-theoretic results do not directly yield explicit information on the invariant distribution or on convergence rate of the process.
Although it is plausible that the algorithm could be studied within a sufficiently refined framework of random dynamical systems, doing so would require technical machinery and specialized terminology.

For this reason, we adopt in the present work a more elementary probabilistic approach based on geometric considerations. This viewpoint allows us to derive explicit formulas describing the evolution of the interval length and the distribution of the normalized root, and to complement the analytical results with numerical experiments. To the best of our knowledge, the probabilistic properties established in this paper have not been previously reported in the literature.

\subsection{Outline of the paper}

The main objective of this paper is to determine explicit expected convergence rates for the stochastic bisection algorithm. To this end, we progressively analyze the probabilistic structure of the random process generated by the algorithm under increasingly general assumptions.

In \Cref{sec:2}, we first consider the setting where the a priori distribution of the root is uniform and the cut distribution is arbitrary. This case serves as a fundamental step toward understanding the convergence behavior of the algorithm. Using elementary probabilistic arguments and geometric interpretations, we show that all normalized roots are identically distributed, explicitly describe the distribution of the scaling factor $\ell_n$ and its first moments in terms of the cut distribution, and obtain formulas for the expected contraction of the interval. Together, these results show that random perturbations of the classical bisection rule give rise to scale-invariant stochastic dynamics: the normalized root position converges to a uniform invariant distribution, allowing explicit computation of the algorithm’s expected convergence rate.

In \Cref{sec:3}, we remove the assumption of a uniform root distribution. We prove convergence for absolutely continuous initial root distributions under the trivial condition $\mathbb{P}[0<c<1]>0$, and establish an explicit rate for this convergence. In particular, this shows that the rates derived in \Cref{sec:2} describe the asymptotic behavior of the algorithm in full generality.

In \Cref{sec:4}, we study a natural generalization of the method, namely a stochastic multisection algorithm in which several random cuts are chosen at each iteration. For simplicity, we restrict our attention to the case where both the root and the cuts are uniformly distributed. We show that the uniform distribution remains stationary for this extended process and derive explicit expected convergence rates as functions of the number of cuts.

Lastly, \Cref{sec:5} presents numerical experiments to validate the theoretical predictions of the previous sections. By considering various features of the root and cut distributions (such as asymmetry or concentration of mass), these experiments provide insight into the robustness of the convergence behavior and illustrate general trends in the algorithm's performance.

\section{Analysis under a uniform prior}\label{sec:2}

\subsection{The interval length after a single iteration}\label{sec:2.1}

In theory, for a given function $f$, the root $r$ sought by the stochastic bisection algorithm is fixed and deterministic. In practice, however, $r$ is unknown and thus behaves like a random variable with an unknown distribution on $(0,1)$ over several runs of the algorithm.

For simplicity and for computational reasons that shall soon become clear, we begin by assuming that the initial root $r$ is uniform on $(0,1)$, that is $r=r_0\sim \mathcal{U}(0,1)$. Moreover, we also assume that the cuts $c_n$ are chosen independently from the normalized roots $r_n$ and that they have the CDF $F$. Following \Cref{alg:RRbisection}, it is easy to deduce that the length of the interval after the first cut is
\begin{equation}\label{def - l}
    \ell_1 = \ell(c_0,r_0)=
    \begin{cases}
    c_0 & ~\text{if }\,c_0\geq r_0,\\
    1-c_0 & ~\text{if }\,c_0<r_0.
    \end{cases}
\end{equation}
Knowing $r_0$, the expected value of $\ell_1$ is then given by:
\begin{align}
    \mathbb{E}[\ell_1\,|\, r_0] &= \mathbb{E}[\ell_1\,|\, c_0>r_0,r_0]\,\mathbb{P}[c_0\geq r_0\,|\, r_0] + \mathbb{E}[\ell_1\,|\, c_0<r_0,r_0]\,\mathbb{P}[c_0<r_0\,|\, r_0] \label{eq:El1r0} \\
    &= \mathbb{E}[c_0 \mathbf{1}_{\{c_0\geq r_0\}}\,|\, r_0] + \mathbb{E}[(1-c_0)\mathbf{1}_{\{c_0< r_0\}}\,|\, r_0] \nonumber \\
    &= \int_{r_0}^1 c_0 \,\mathrm{d} F(c_0) + \int_0^{r_0} (1-c_0) \,\mathrm{d} F(c_0). \nonumber
\end{align}
By the law of total expectation, taking the expected value of this expression over the distribution of $r_0$ will give the expected scaling factor after the first cut.

\begin{ex}\label{example - unif}
    Knowing, for instance, that $c\sim \mathcal{U}(0,1)$, $\mathrm{d} F(c_0)= \mathrm{d} c_0$ and we get
\begin{equation*}
   \mathbb{E}[\ell_1\,|\, r_0] = \frac{1 + 2r_0 - 2 r_0^2}{2}.
\end{equation*}

The expected length of the interval after the first random cut depends on the value of the root. As seen in \Cref{fig:explength}, it varies between $\frac{1}{2}$ (if the root is near the endpoints) and $\frac{3}{4}$ (if the root is near the center). By integrating over $r_0$, we find $\mathbb{E}[\ell_1] = \frac{2}{3}$.
\end{ex}

\begin{ex}
    Assuming that $c \sim \mathrm{Beta(2,2)}$, we get
    \begin{align*}
        \mathbb{E}[\ell_1\,|\, r_0] =\biggl(\frac{1}{2} - 2r_0^3 + \frac{3}{2}r_0^4\biggr) + \biggl(3r_0^2 - 4r_0^3+\frac{3}{2}r_0^4\biggr).
    \end{align*}
    The distribution of the expected scaling factor with respect to $r_0$ is shown in \Cref{fig:explength_beta22}. After integrating, the scaling factor is $\mathbb{E}[\ell_1] = \frac{3}{5}$. Intuitively, the scaling factor is smaller than in the previous example (uniform cut) because the distribution has a smaller dispersion ($\sigma^2 = \frac{1}{20}$ as opposed to $\sigma^2=\frac{1}{12}$) around the mean. Thus, the algorithm is ``closer'' to the deterministic bisection method (scaling factor of exactly $\frac{1}{2}$).
\end{ex}

\begin{figure}[ht]
\begin{minipage}{.49\linewidth}
    \vspace{-2.5pt}
    \begin{figure}[H]
    \centering
    \begin{tikzpicture}[scale=.7]
        \begin{axis}[
            width=9cm,
            height=8cm,
            xlabel={$r_0$},
            ylabel={$\mathbb{E}[\ell_1|r_0]$},
            ymax=.755,
            domain=0:1,
            samples=200,
            axis lines=left,
            grid=both,
        ]
        \addplot[
            thick,
        ]
        {(1 + 2*x - 2*x^2)/2};
        \end{axis}
    \end{tikzpicture}
    \vspace{-6pt}
    \caption{Expected length after one iteration if $c_0\sim \mathcal{U}(0,1)$.}
    \label{fig:explength}
\end{figure}
\end{minipage}%
\begin{minipage}{.02\linewidth}
\hfill
\end{minipage}%
\begin{minipage}{.49\linewidth}
    \begin{figure}[H]
    \centering
    \begin{tikzpicture}[scale=.7]
        \begin{axis}[
            width=9cm,
            height=8cm,
            ymax=.69,
            xlabel={$r_0$},
            ylabel={$\mathbb{E}[\ell_1|r_0]$},
            domain=0:1,
            samples=200,
            axis lines=left,
            grid=both,
        ]
        \addplot[
            thick,
        ]
        {0.5 - 2*x^3 + 1.5*x^4) + 3*x^2 - 4*x^3+1.5*x^4)};
        \end{axis}
    \end{tikzpicture}
    \vspace{-6pt}
    \caption{Expected length after one iteration if $c_0\sim \mathrm{Beta}(2,2)$.}
    \label{fig:explength_beta22}
\end{figure}
\end{minipage}
\end{figure}

Notice that, after one step of the algorithm and assuming that $r_0$ is uniform on $(0,1)$, $\mathbb{E}[\ell_1]$ depends then only on the distribution of the cut. If it was the case that the distribution of the root $r_1$ after one step of the algorithm was still uniform on $(0,1)$ after the rescaling, then the arguments above would remain valid and $\mathbb{E}[\ell_n]$ would stay constant throughout the iterations. 

\subsection{Stationarity of the Uniform Distribution}

We show here that, assuming that $r_n\sim \mathcal{U}(0,1)$, the distribution of the root $r_{n+1}$ is indeed still uniform after the rescaling. 

\begin{prop}
    If $r_0\sim\mathcal{U}(0,1)$, then $r_n\sim \mathcal{U}(0,1)$ for all $n\geq1$. That is, $\mathcal{U}(0,1)$ is a stationary distribution for the random process described in \Cref{alg:RRbisection}. \label{thm:stationarity}
\end{prop}
\begin{proof}

The rescaled root $r_1=T(c_0,r_0)$ will be uniformly distributed if $\mathbb{P}[T(c_0,r_0) \leq t] = t$ for all $t\in[0,1]$. Let $L^-_t(T)=\{(c_0,r_0): T(c_0,r_0) \leq t\}$ be the $t$-lower level set of $T$, which admits the geometric description\vspace{-.6pt}
\begin{equation}\label{eq - region}
    \bigl\{(c_0,r_0): c_0\in[0,1], r_0\in[0,tc_0]\bigr\} \cup  \bigl\{(c_0,r_0):  c_0\in[0,1],r_0\in (c_0,t +(1-t)c_0] \bigr\},\vspace{-.5pt}
\end{equation}
and is illustrated in \Cref{fig:regionLT}. Therefore, it follows from the definition that\vspace{-.5pt}
\begin{align*}
    \mathbb{P}[T(c_0,r_0) \leq t] &= \int_0^1\int_0^1 \mathbf{1}_{L^-_t(T)} \,\mathrm{d} r_0 \,\mathrm{d} F(c_0)\\
    &=\int_0^1  \bigg(\int_0^{tc_0} \mathrm{d} r_0 + \int_{c_0}^{t+(1-t)c_0} \mathrm{d} r_0 \biggr) \,\mathrm{d} F(c_0) \\
    & = t \int_0^1 \mathrm{d} F(c_0) \;=\; t.\vspace{-.5pt}
\end{align*}
Since the argument is independent of $n$, we may recursively show that $r_n \sim \mathcal{U}(0,1)$ for any $n\geq 1$, thus proving that $\mathcal{U}(0,1)$ is a stationary distribution of the algorithm.
\end{proof}

Note that the fact that $\mathcal{U}(0,1)$ is a stationary distribution of the random process described in \Cref{alg:RRbisection} holds no matter the CDF $F$. Hence, at each step of the process, the rescaled root $r_n$ always follows a uniform distribution on $[0,1]$, regardless of the distribution chosen for the cut. The following corollary is obtained as a consequence of what was presented in \Cref{sec:2.1}.

\begin{cor}\label{cor - constant}
    Let $r\sim \mathcal{U}(0,1)$. There exists a constant $C>0$ such that the random variables $\ell_n$ generated by \Cref{alg:RRbisection} satisfy $\mathbb{E}[\ell_n] = C$ for all $n\geq 1$, with the constant $C$ depending only upon the chosen distribution $\mathcal{D}$ of the cuts.
\end{cor}

For instance, if $\mathcal{D} = \mathcal{U}(0,1)$, then it follows from \Cref{example - unif} that $\mathbb{E}[\ell_n] = \frac{2}{3}$ for all $n\geq 1$.

\subsection{The distribution of \texorpdfstring{$\ell_n$}{ℓₙ}}

\Cref{cor - constant} reveals that the expected length $\mathbb{E}[\ell_n]$ is always constant for all $n\geq 1$. We seek to determine a general formula for $\mathbb{E}[\ell_n]$ depending on an arbitrary cut distribution, i.e.,
\begin{equation*}
    \mathrm{cut}(0,1) = c \sim \mathcal{D}(0,1).
\end{equation*}
To do so, we begin by establishing a formula for the CDF and PDF (whenever it exists) of the random variables $\ell_n$. 
In particular, we show that the random variables $\ell_n$ have a distribution which is closely related to the one of the cuts $c_n$. 

\begin{thm}\label{thm - pdf}
    Suppose that $r \sim \mathcal{U}(0,1)$ and $c$ is a random variable with cumulative distribution function $F$. Define $\ell$ as in \eqref{def - l}. 
    Then the cumulative distribution function of $\ell$ is
    \[
    H(t)
    =
    \int_0^t x \,\mathrm{d} F(x)
    +
    \int_{1-t}^1 (1-x)\,\mathrm{d} F(x)
    \]
    Moreover, if $c$ has a probability density function $f(c)$, then $\ell$ has the probability density function
    \[
    h(t) = t\bigl(f(t)+f(1-t)\bigr).
    \]
\end{thm}
\begin{proof}
We first compute the cumulative distribution function of $\ell$. Let $t\in[0,1]$. By definition, we have $H(t)=\mathbb{P}[\ell\leq t]$. Now, since
\[
\ell=\begin{cases}
c & ~\text{if }\,c\geq r,\\
1-c & ~\text{if }\,c<r,
\end{cases}
\]
we have
\[
\{(c,r): \ell \leq t \} = \{(c,r): r\leq c \leq t\} \cup  \{(c,r):  1-t \leq c < r\},
\]
and these two events are disjoint (see \Cref{fig:region2}). Hence
\[
H(t)
=
\mathbb{P}[r\leq c\leq t]+\mathbb{P}[1-t\leq c < r].
\]
Since $r\sim \mathcal{U}(0,1)$, we have $\mathbb{P}[r\leq x\,|\, c=x] = x.$ Therefore,
\[
\mathbb{P}[r\leq c\leq t]
=
\int_0^t \mathbb{P}[r\leq x\,|\, c=x] \,\mathrm{d} F(x)
=
\int_0^t x \, \mathrm{d} F(x).
\]
Conditioning again on $c=x$, we have $\mathbb{P}[x\leq r\,|\, c=x]=1-x,$ from which it follows that
\[
\mathbb{P}[1-t \leq c < r]
=
\int_{1-t}^1 \mathbb{P}[x<r\,|\, c=x] \,\mathrm{d} F(x)
=
\int_{1-t}^1 (1-x)\,\mathrm{d} F(x).
\]
Combining the two parts then yields
\[
H(t)
=
\int_0^t x \,\mathrm{d} F(x)
+
\int_{1-t}^1 (1-x)\,\mathrm{d} F(x),
\qquad 0\le t\le 1.
\]
Moreover, assuming that $c$ has the PDF $f$, we have $\mathrm{d} F(x) = f(x)\,\mathrm{d} x$ and we finally find by differentiating with respect to $t$ that $\ell$ has the PDF
\begin{align*}
    h(t) = H'(t) = \frac{d}{dt} \bigg( \int_0^t x f(x)\,\mathrm{d} x
+
\int_{1-t}^1 (1-x)f(x)\,\mathrm{d} x \biggr) = t f(t) + t f(1-t). \qedhere
\end{align*}
\end{proof}

\begin{figure}[ht]
\begin{minipage}[t]{.48\linewidth}
\begin{figure}[H]
    \centering
    \begin{tikzpicture}[scale=5.1]
    
    \definecolor{bluefill}{RGB}{25,25,112}
    \definecolor{redfill}{rgb}{0.8, 0.25, 0.33}
    
    \def\t{0.4}
    
    \fill[bluefill, opacity=0.35]
        (0,\t) -- (1,1) -- (0,0) -- cycle;
    
    \fill[redfill, opacity=0.35]
        (0,0) -- (1,\t) -- (1,0) -- cycle;
    
    \draw[thick] (0,0) rectangle (1,1);
    
    \draw[bluefill, dashed, very thick] (1,1) -- (0,0);
    \draw[bluefill, very thick] (1,1) -- (0,\t);
    \draw[redfill, very thick] (0,0) -- (1,\t);
    
    \fill (0,\t) circle (0.01) node[left] {$(0,t)$};
    \fill (1,\t) circle (0.01) node[right] {$(1,t)$};
    
    \draw[thick] (-0.03,0) -- (1.05,0);
    \draw[thick] (0,-0.03) -- (0,1.05);

    \draw (1.05,0) node[right] {$c_0$};
    \draw (0,1.05) node[above] {$r_0$};
    
    \foreach \x/\lab in {0/0,0.5/0.5,1/1}
        \draw (\x,0) -- (\x,-0.012) node[below] {\lab};
    
    \foreach \y/\lab in {0.5/0.5,1/1}
        \draw (0,\y) -- (-0.012,\y) node[left] {\lab};
    
    \end{tikzpicture}
    \caption{ The region $L^-_t(T)$. The blue section corresponds to $r_0 > c_0$, and the red region corresponds to $r_0\leq c_0$.}
    \label{fig:regionLT}
\end{figure}
\end{minipage}%
\begin{minipage}{.04\linewidth}
\hfill
\end{minipage}%
\begin{minipage}[t]{.48\linewidth}
\begin{figure}[H]
    \centering
    \begin{tikzpicture}[scale=5.1]
    
    \definecolor{bluefill}{RGB}{25,25,112}
    \definecolor{redfill}{rgb}{0.8, 0.25, 0.33}
    
    \def\t{0.35}
    
    \fill[bluefill, opacity=0.35]
        (1-\t,1-\t) -- (1,1) -- (1-\t,1) -- cycle;
    
    \fill[redfill, opacity=0.35]
        (0,0) -- (\t,\t) -- (\t,0) -- cycle;
    
    \draw[thick] (0,0) rectangle (1,1);
    
    \draw[bluefill, very thick] (1-\t,1-\t) --  (1-\t,1);
    \draw[bluefill, dashed, very thick] (1,1) -- (1-\t,1-\t);
    \draw[redfill, very thick] (\t,0) -- (\t,\t);
    \draw[redfill, very thick] (\t,\t) -- (0,0);
    \draw[gray, dashed, thick] (1-\t,1-\t) -- (\t,\t);
    
    \fill (\t,0) circle (0.01) node[below] {$(t,0)$};
    \fill (1-\t,1) circle (0.01) node[above] {$(1-t,1)$};
    
    \draw[thick] (-0.03,0) -- (1.05,0);
    \draw[thick] (0,-0.03) -- (0,1.05);

    \draw (1.05,0) node[right] {$c$};
    \draw (0,1.05) node[above] {$r$};
    
    \foreach \x/\lab in {0/0,0.5/0.5,1/1}
        \draw (\x,0) -- (\x,-0.012) node[below] {\lab};
    
    \foreach \y/\lab in {0.5/0.5,1/1}
        \draw (0,\y) -- (-0.012,\y) node[left] {\lab};
    
    \end{tikzpicture}
    \caption{The region $\{(c,r): \ell \leq t \}$. The blue section corresponds to $r>c$, and the red region corresponds to $r\leq c$.}
    \label{fig:region2}
\end{figure}
\end{minipage}
\end{figure}

Recall that our main goal in this section was to compute $\mathbb{E}[\ell_n]$ for an arbitrary cut distribution. As a corollary to the previous theorem, we obtain an explicit formula for this quantity, as well as the variance of $\ell_n$. Surprisingly, we may express neatly both the mean and variance of the random variable $\ell$ as simple polynomials of the quantity $\mu -\mu^2 - \sigma^2 = \mathbb{E}[c(1-c)]$. 

\begin{cor}\label{cor - mean_variance_special}
    Suppose that $r \sim \mathcal{U}(0,1)$ and $c$ is a random variable with mean $\mu$ and variance $\sigma^2$. Then 
    \begin{equation}\label{eq - def mu ell}
    \begin{gathered}
        \mu_\ell = 1-2(\mu -\mu^2 - \sigma^2),\\
        \sigma_\ell^2 = (\mu -\mu^2 - \sigma^2) \bigl(1-4 (\mu -\mu^2 - \sigma^2)\bigr).
    \end{gathered}
    \end{equation}
    In particular, if $\mu=\frac{1}{2}$, then $\mu_\ell = \frac{1}{2}+2\sigma^2$ and $\sigma_\ell^2 = \sigma^{2}(1-4\sigma^{2})$.
\end{cor}
\begin{proof}
    Since $H(t) = \int_0^t x \,\mathrm{d} F(x) + \int_{1-t}^1 (1-x)\,\mathrm{d} F(x)$, it follows from a simple change of variable and the Fundamental Theorem of Calculus that
    \[
    \mathrm{d} H(t) = t \,\mathrm{d} F(t) + t\,\mathrm{d} F(1-t).
    \]
    Hence, \Cref{thm - pdf} and a second change of variables yields
    \begin{align*}
        \mu_\ell &= \int_0^1 x \,\mathrm{d} H(x) = \int_0^1 x^2 \,\mathrm{d} F(x) + \int_0^1 x^2 \,\mathrm{d} F(1-x) \\
        &= \int_0^1 x^2 \,\mathrm{d} F(x) + \int_0^1 (1-x)^2 \,\mathrm{d} F(x) = \int_0^1 \bigl(1-2x+2x^2\bigr) \,\mathrm{d} F(x).
    \end{align*}
    Since $\int_0^1 x^2 \,\mathrm{d} F(x) = \mu^2 + \sigma^2$, it follows that
    \begin{align*}
        \mu_\ell &= 1-2\mu+2(\mu^2+\sigma^2) = 1-2(\mu-\mu^2-\sigma^2).
    \end{align*}
    For the variance, we obtain with similar arguments as before that
    \begin{align*}
        \sigma_\ell^2+\mu_\ell^2 &= \int_0^1 x^2 \,\mathrm{d} H(x) = \int_0^1 x^3\,\mathrm{d} F(x) + \int_0^1 (1-x)^3\,\mathrm{d} F(x) \\
        &= \int_0^1 \bigl(3 x^2 - 3 x + 1 \bigr)\,\mathrm{d} F(x) = 3(\mu^2+\sigma^2)-3\mu+1.
    \end{align*}
    Therefore, it finally follows that
    \begin{align*}
        \sigma_\ell^2 &= 3(\mu^2+\sigma^2)-3\mu+1 - \mu_\ell^2 \\
        &= 3(\mu^2+\sigma^2)-3\mu+1 - (1-2\mu+2\mu^2+2\sigma^2)^2 \\
        &= (\mu -\mu^2 - \sigma^2) \bigl(1-4 (\mu -\mu^2 - \sigma^2)\bigr). \qedhere
    \end{align*}
\end{proof}

Remark that in the particularly natural case where the cuts are taken from a symmetric distribution, we obtain the elegant result that
\[
\mathbb{E}[\ell_{n}] = \frac{1}{2}+2\sigma^2, \qquad n\geq 1.
\]
In particular, this shows that for symmetric distributions, the length of the interval at each step of the algorithm is reduced on average by a factor of $\frac{1}{2}+2\sigma^2$, which is minimal when $\sigma^2 \to 0$. This corresponds precisely to the case of the classical bisection algorithm. \Cref{fig:muell} shows graphically the relations for $\mu_\ell$ and $\sigma^2_\ell$ in this case.

More generally, observe that $0 \le c(1-c) \le \frac14$ for all $c\in[0,1]$. Since $1-2\mu+2\mu^2+2\sigma^2 = 1-2\,\mathbb{E}[c(1-c)]$, taking expectations thus yields
\[
\frac{1}{2} \le \mu_\ell \le 1,
\]
with equality above if and only if $c\in\{0,1\}$ almost surely, and equality below if and only if $c=\frac{1}{2}$ almost surely (i.e., the classical bisection algorithm). Therefore, the quantity $\mathbb{E}[\ell_n]$ is minimized \emph{uniquely} (up to a set of measure zero) by the classical bisection algorithm.

\begin{cor}
    The classical bisection algorithm is optimal in the family of stochastic bisection algorithms.
\end{cor}

\subsection{Properties of \texorpdfstring{$L_n$}{Lₙ}}

The random variable $L_n$ generated by \Cref{alg:RRbisection} corresponds to the length of the working interval after $n$ iterations. We are interested in the expected value of this length. To do so, some independence property has to be shown first to compute explicitly this expectation.

\begin{thm}\label{thm - ind}
    Given $r\sim \mathcal{U}(0,1)$, then:
    \begin{itemize}
        \item[(i)] $r_n$ is independent from $\ell_n$, for all $n>0$;
        \item[(ii)] the random variables $\ell_1,\ell_2,\dots,$ are mutually independent.
    \end{itemize}
\end{thm}
\begin{proof}
We seek to compute the probability $\mathbb{P}[r_1 \leq t,\, \ell_1 \leq u]$ and show that the joint distribution of $r_1$ and $\ell_1$ is $\mathbb{P}[r_1 \leq t] \,\mathbb{P}[\ell_1 \leq  u]$. Since $\{r_0\leq c_0\}$ and $\{r_0>c_0\}$ are complementary events, we can partition the probability in two parts:
\begin{align*}
    \mathbb{P}[r_1 \leq t,\, \ell_1 \leq u] &= \mathbb{P}[r_1 \leq t,\, \ell_1 \leq u ,\, r_0\leq c_0] + \mathbb{P}[r_1 \leq t,\, \ell_1 \leq u ,\, r_0>c_0] \\
    &= \mathbb{P}[c_0 \leq u,\, r_0 \leq tc_0] + \mathbb{P}[1-u \leq c_0 ,\, r_0 \leq t(1-c_0)+ c_0 ,\, r_0 > c_0] .
\end{align*}
The first term bounds a triangular region $\{c_0\in [0,u],\, r_0 \in [0, tc_0]\}$ (\Cref{fig:regionthm4}). Since $c_0$ and $r_0$ are independent, and since $r_0\sim \mathcal{U}(0,1)$, integrating over this region yields\vspace{-2pt}
\begin{align*}
    \mathbb{P}[c_0 \leq u,\, r_0 \leq tc_0] &= \int_0^u \mathbb{P}[ r_0 \leq tc_0 \,|\, c_0] \,\mathrm{d} F(c_0) = t\int_0^{u}  c_0 \,\mathrm{d} F(c_0).\vspace{-2pt}
\end{align*}
Similarly, the second term bounds the triangular region $\{c_0\in[1-u,1],\, r_0\in(c_0,t(1-c_0)+c_0] \}$. Since $r_0 $ is uniform on $(0,1)$, we have $\mathbb{P}[c_0 <r_0 < t(1-c_0)+c_0 \,|\, c_0] = t(1-c_0)$ and thus\vspace{-2pt}
\begin{align*}
    \mathbb{P}[1-u \leq c_0 ,\, r_0 \leq t(1-c_0)+ c_0 ,\, r_0 > c_0] &= \int_{1-u}^1 \mathbb{P}[c_0 < r_0 \leq t(1-c_0)+c_0 \,|\, c_0] \,\mathrm{d} F(c_0)\\
    &= t\int_{1-u}^1(1-c_0) \,\mathrm{d} F(c_0).\\[-22pt]
\end{align*}
Therefore, the total probability becomes\vspace{-2pt}
\begin{align*}
    \mathbb{P}[r_1 \leq t,\, \ell_1 \leq u] = t\bigg(\int_0^{u} c_0 \,\mathrm{d} F(c_0) + \int_{1-u}^1(1-c_0) \,\mathrm{d} F(c_0) \biggr) = tH(u) = \mathbb{P}[r_1 \leq t] \,\mathbb{P}[\ell_1 \leq u],
\end{align*}
where $H$ is the CDF of $\ell$ taken from \Cref{thm - pdf}. Thus, $r_1$ is independent from $\ell_1$. Moreover, for all $n>0$,
\begin{enumerate}
    \item $r_n \sim \mathcal{U}(0,1)$ (\Cref{thm:stationarity});
    \item $c_n$ has CDF $F$;
    \item $c_n$ is independent from $r_n$.
\end{enumerate}
Hence, the computation above can be repeated, mutatis mutandis, with ($r_{n+1}$, $\ell_{n+1}$) instead of ($r_1$, $\ell_1$), showing that $r_{n+1}$ and $\ell_{n+1}$ are independent for all $n>0$. 

From there, we can form an argument for the independence of all $\ell_n$ based on strong induction. Assume that $r_n$ is jointly independent from $(\ell_i)_{i=1}^n$. The base case is validated by the computations above. Moreover, $c_n$ is also clearly jointly independent from the tuple of all $\ell_i$, meaning $(c_n,r_n)$ is jointly independent from $(\ell_i)_{i=1}^n$. Now, $\ell_{n+1}=\ell({c_n,r_n})$ depends only upon $r_n$ and $c_n$ in its construction.  Therefore, for any measurable $U,V$,
\begin{align*}
    \mathbb{P}[\ell_{n+1} \in U, (\ell_1,\dots,\ell_n) \in V] = \mathbb{P}[\ell(c_n,r_n)\in U]\mathbb{P}[ (\ell_1,\dots,\ell_n) \in V].
\end{align*}
Thus, $\ell_{n+1}$ is independent of $(\ell_1,\dots,\ell_n) $. Since $\ell_1,\dots,\ell_n$ are mutually independent by the induction hypothesis, it follows that $\ell_1,\dots,\ell_{n+1}$ are mutually independent for all $n>0$.
\end{proof}

\begin{figure}[ht]
\begin{minipage}[t]{.48\linewidth}
\vspace{19pt}
\begin{figure}[H]
    \centering
    \begin{tikzpicture}[scale=.8]
    \definecolor{bluefill}{RGB}{25,25,112} 
    \definecolor{redfill}{rgb}{0.8, 0.25, 0.33} 
    
    \begin{axis}[
        width=9cm, height=8.2cm,
        axis lines = middle,
        x label style={at={(axis description cs:1.01,.1)},anchor=west},
        xlabel = {$\sigma^2$},
        xmin = -0.03, xmax = 0.28,
        ymin = -0.09, ymax = 1.1,
        xtick={0,0.05,0.1,0.15,0.20,0.25},
        xticklabels={0,0.05,0.1,0.15,0.20,0.25},
        samples = 400,
        domain = 0:0.25,
    ]
    \addplot[bluefill, very thick] {1/2 + 2*x};
    \addplot[redfill, very thick] {sqrt(x*(1-4*x))};

    \node at (axis cs:.125,.3) {$\sigma_\ell$};
    \node at (axis cs:.125,.68) {$\mu_\ell$};
    \end{axis}
    \end{tikzpicture}
    \caption{The mean $\mu_\ell$ and standard deviation $\sigma_\ell$ of 
    $\ell=\ell(c,r)$ in terms of the variance $\sigma^2$ of the \emph{symmetric} random variable $c$.}
    \label{fig:muell}
\end{figure}
\end{minipage}%
\begin{minipage}{.04\linewidth}
\hfill
\end{minipage}%
\begin{minipage}[t]{.48\linewidth}
    \begin{figure}[H]
    \centering
    \begin{tikzpicture}[scale=4.8]
    
    \definecolor{bluefill}{RGB}{25,25,112} 
    \definecolor{redfill}{rgb}{0.8, 0.25, 0.33} 
    
    \def\t{0.4}
    \def\u{0.35}
    
    \fill[bluefill, opacity=0.35]
        (1-\u,\t*\u+1-\u) -- (1,1) -- (1-\u,1-\u) -- cycle;
    
    \fill[redfill, opacity=0.35]
        (0,0) -- (\u,\t*\u) -- (\u,0) -- cycle;
    
    \draw[thick] (0,0) rectangle (1,1);
    
    \draw[bluefill, dashed, very thick] (1-\u,\t*\u+1-\u) -- (1-\u,1-\u) -- (1,1) -- cycle;
    \draw[redfill, dashed, very thick] (0,0) -- (\u,\t*\u) -- (\u,0);

    \draw[lightgray, dashed, thick] (1,\t) -- (\u,\t*\u);
    \draw[lightgray, dashed, thick] (0,0) -- (1-\u,1-\u);
    \draw[lightgray, dashed, thick] (0,\t) -- (1-\u,\t*\u+1-\u);
    \draw[lightgray, dashed, thick] (1-\u,1) -- (1-\u,\t*\u+1-\u);

    \fill (1-\u,1) circle (0.01) node[above] {$(1-u,1)$};
    \fill (0,\t) circle (0.01) node[left] {$(0,t)$};
    \fill (1,\t) circle (0.01) node[right] {$(1,t)$};
    \fill (\u,0) circle (0.01) node[below] {$(u,0)$};

    \draw[thick] (-0.03,0) -- (1.05,0);
    \draw[thick] (0,-0.03) -- (0,1.05);

    \draw (1.05,0) node[right] {$c_0$};
    \draw (0,1.05) node[above] {$r_0$};
    
    \foreach \x/\lab in {0/0,0.5/0.5,1/1}
        \draw (\x,0) -- (\x,-0.012) node[below] {\lab};
    
    \foreach \y/\lab in {0.5/0.5,1/1}
        \draw (0,\y) -- (-0.012,\y) node[left] {\lab};
    
    \end{tikzpicture}
    \caption{The regions in the integrals of \Cref{thm - ind}. Note that they correspond to the intersection of \Cref{fig:regionLT,fig:region2}.}
    \label{fig:regionthm4}
\end{figure}
\end{minipage}
\end{figure}

\begin{cor}\label{cor - final}
The expected length of the $n$-th interval is
\begin{equation*}
    \mathbb{E}[L_n] = \mathbb{E}[\ell_1]^n.
\end{equation*}
\end{cor}
\begin{proof}
    Since the $\ell_i$ ($i>0$) are all independent, it directly follows from \Cref{cor - constant} that
    \begin{equation*}
        \mathbb{E}[L_n] = \mathbb{E}\bigg[\prod_{i=1}^n \ell_i\biggr] = \prod_{i=1}^n \mathbb{E}[\ell_i] = \mathbb{E}[\ell_1]^n. \qedhere
    \end{equation*}
\end{proof}

This corollary implies that on average, when $r\sim\mathcal{U}(0,1)$, the length of the interval after $n$ iterations has length $C^n$, just like the deterministic bisection method with $C=1/2$.

\section{Arbitrary root distribution}\label{sec:3}

    We now study the case where the initial root $r_0$ is not uniformly distributed. In this case, the roots $r_n$ need not share the same distribution, which complicates each step of the above analysis. Nevertheless, under mild hypotheses, we recover essentially the same conclusion by showing that the distribution of $r_n$ converges uniformly, in the appropriate sense, to the uniform distribution.

    Indeed, it is easily seen that the conclusion does not hold in full generality. Consider for instance the case where the cuts $c_n$ are almost surely $0$ or $1$. Then at each step of the process, $r_{n+1}$ is almost surely equal to $r_n$. Hence, if $r_0$ is not initially uniform, then it is clearly impossible for the root to converge to the uniform distribution if $\mathbb{P}[0<c<1]=0$.

    However, even if the cuts are chosen according to a distribution satisfying $\mathbb{P}[0<c<1]>0$, the roots may still not converge to the uniform distribution. Consider for instance the case where $c_n$ is almost surely equal to $1/2$ and $r_0=1/3$ with probability $p$ and $r_0=2/3$ with probability $1-p$. Then it is easily seen that $r_1$ is then equal to $2/3$ with probability $p$ and $1/3$ with probability $1-p$. Proceeding iteratively, one then concludes that
    \[
    \mathbb{P}\bigl[r_n=\tfrac{1}{3}\bigr] = \begin{cases}
        p &\text{if } n\text{ is even},\\
        1-p &\text{if } n\text{ is odd},
    \end{cases} \qquad\&\qquad \mathbb{P}\bigl[r_n=\tfrac{2}{3}\bigr] = \begin{cases}
        1-p &\text{if } n\text{ is even},\\
        p &\text{if } n\text{ is odd}.
    \end{cases} 
    \]
    Once again, the distribution of $r_n$ does not converge to the uniform distribution.

    The previous examples were somewhat artificial and of limited practical use. Hence, one may naturally conjecture that the convergence to the uniform distribution will occur if we wander outside of these pathological cases, for instance when $r_0$ is not a discrete distribution. However, this is also not the case. Indeed, consider the case where $c_n$ is almost surely $1/2$ and
    \begin{equation}\label{eq - weird}
        r_0 = \sum_{k=1}^\infty 2^{-k}B_k,
    \end{equation}
    with $B_k$ being i.i.d. Bernoulli$(p)$ random variables with $p\neq 1/2$. Equivalently, if $r_0$ is written in binary form $0.b_1b_2b_3\dots$, then the digits $b_i$ are chosen independently with $\mathbb{P}(b_i=1)=p$ and $\mathbb{P}(b_i=0)=1-p$. It is easily seen that the CDF of $r_0$ is continuous and thus that $r_0$ has no point masses. However, this distribution is not uniform, since
    \[
    \mathbb{E}[r_0] = \sum_{k=1}^\infty 2^{-k}\cdot p = p\neq \frac{1}{2}.
    \]
    The CDF of $r_0$ has a fractal (in particular, self-similar) structure that is shown in \Cref{fig - fractal}. 
    Since for this choice of distribution of $c_n$, $T$ corresponds precisely to the deterministic Dyadic map, it is known that the bits in the binary expansion of $r_n$ will be shifted to the left, that is
    \begin{align*}
        r_n = 0.b_1b_2b_3\dots ~~\mapsto~~ r_{n+1} = 0.b_2b_3b_4\dots
    \end{align*}
    Since they were all i.i.d., we find that the distribution of $r_{n+1}$ is identical to the one of $r_n$. Proceeding inductively, it follows that the distribution of \eqref{eq - weird} is a stationary distribution of the process and thus that it never converges to the uniform distribution.

    \begin{figure}
        \centering
        \includegraphics[width=0.5\linewidth]{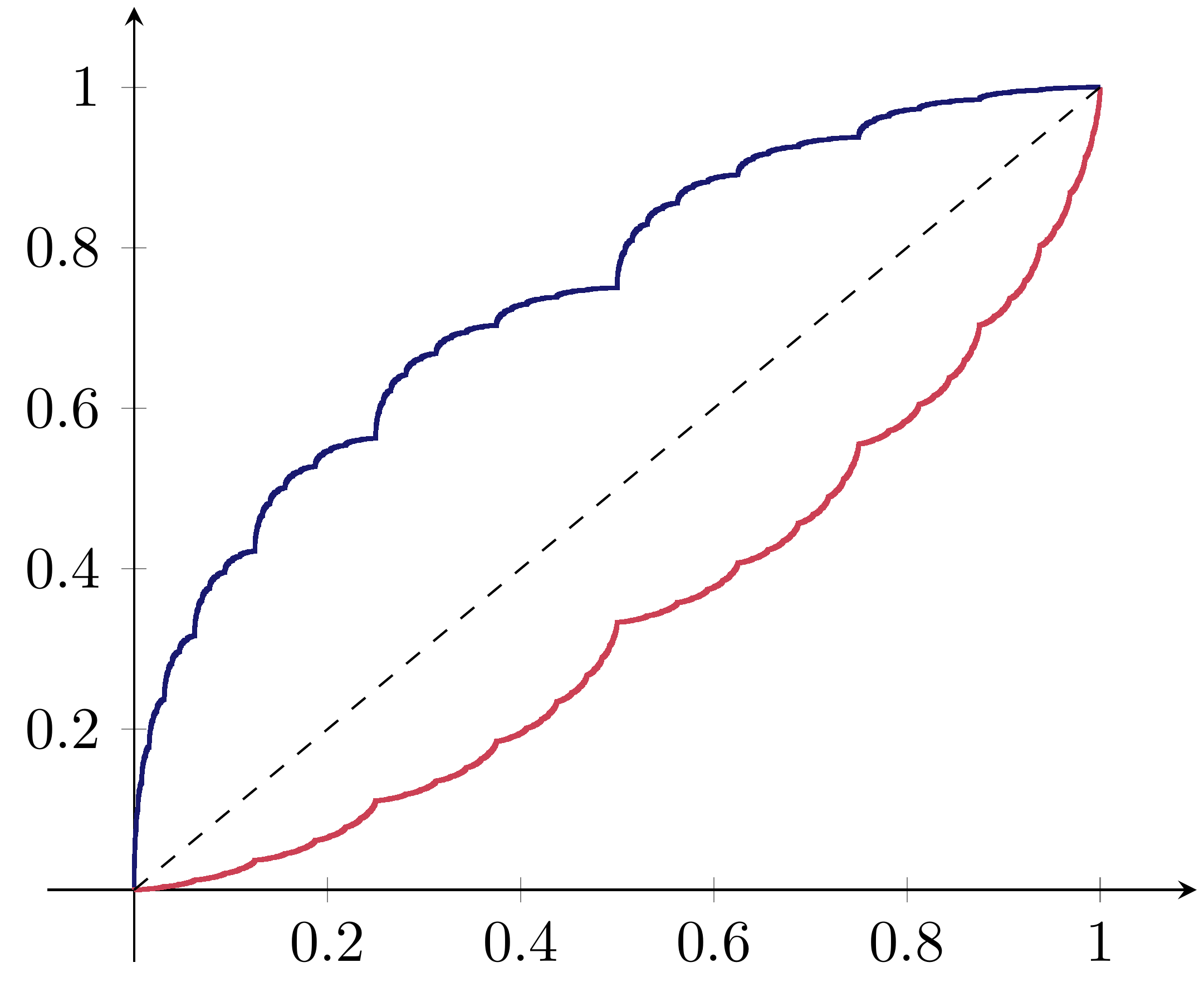}
        \caption{The cumulative distribution function of $r_0$ defined in \eqref{eq - weird} with $p=\frac{1}{4}$ in blue (above the dashed uniform CDF) and $p=\frac{2}{3}$ in red (below the dashed uniform CDF).}
        \label{fig - fractal}
    \end{figure}

    The previous examples show that the distribution of $r_n$ may not converge to the uniform distribution. However, we may show that it always does if $r_0$ is \emph{absolutely continuous}, i.e., if it admits a probability density function. Hence, denote by $F$ the CDF of the cuts $c_n$, by $G_n$ the CDF of the rescaled roots $r_n$, and assume that $r_0$ has PDF $g_0$, that is\vspace{-1pt}
    \[
    G_0(t)=\int_0^t g_0(u)\,\mathrm{d}u .\vspace{-1pt}
    \]
    Proceeding as in the proof of \Cref{thm:stationarity}, the CDF $G_{n+1}$ of $r_{n+1}$ is given by\vspace{-1pt}
    \begin{align}
        G_{n+1}(t) = (\mathcal{T}G_n)(t) = \int_0^1
        \bigl(G_n(tc)+G_n(t+(1-t)c)-G_n(c)\bigr)
        \,\mathrm{d}F(c).\vspace{-1pt}
        \label{eq - operator}
    \end{align}
    If $G_n$ is absolutely continuous with density $g_n$, then\vspace{-1pt}
    \begin{align*}
        G_{n+1}(t)
        &= \int_0^1
        \!\left(
            \int_0^{tc} g_n(r)\,\mathrm{d}r
            + \int_c^{t+(1-t)c} g_n(r)\,\mathrm{d}r
        \right)
        \mathrm{d}F(c) \\
        &= \int_0^t
        \int_0^1
        \bigl(
            c g_n(cu)
            +(1-c)g_n(c+(1-c)u)
        \bigr)
        \,\mathrm{d}F(c)\,\mathrm{d}u .\vspace{-1pt}
    \end{align*}
    Hence, $G_{n+1}$ is also absolutely continuous with density\vspace{-1pt}
    \begin{equation}\label{eq - PDF}
        g_{n+1}(u)
        = \mathcal{L}g_n(u)
        := \int_0^1
        \bigl(
            c g_n(cu)
            +(1-c)g_n(c+(1-c)u)
        \bigr)
        \,\mathrm{d}F(c).
    \end{equation}
    Therefore, since $G_0$ is absolutely continuous by hypothesis, it follows by induction that $r_n$ has a PDF $g_n$ for every $n\geq 0$, and that the densities satisfy $g_{n+1}=\mathcal{L}g_n$.

\begin{thm}\label{thm - root_conv}
Let $F$ be the cumulative distribution functions of the cuts $c_n$ and $g_n$ the probability density functions of the absolutely continuous normalized roots $r_n$. If $\mathbb{P}[0<c<1]>0$, then\vspace{-1.41pt}
\[
\|g_n-1\|_{L^1}\longrightarrow 0.\vspace{-1.41pt}
\]
Moreover, we have $\|g_n-1\|_{L^1} \leq \varepsilon + O(\mu_\ell^n)$ for all $\varepsilon>0$, where $\mu_\ell$ is defined in \eqref{eq - def mu ell}.
\end{thm}
\begin{proof}
Define $\psi_n:=g_n-1$. Then $\int_0^1 \psi_n(u)\,\mathrm{d}u=0$ and
$\mathcal{L}^n g_0-1 = \mathcal{L}^n \psi_0$. For $k\ge 1$, define\vspace{-1.41pt}
\[
\lambda_k
:=
\int_0^1 \!\bigl(c^{\ell+1}+(1-c)^{\ell+1}\bigr)\,\mathrm{d}F(c).\vspace{-1.41pt}
\]
Then, since $\mathbb{P}[0<c<1]>0$ and since $c,1-c\in[0,1]$, we have\vspace{-1.41pt}
\[
1=\lambda_0>\lambda_1>\lambda_2>\cdots>0.\vspace{-1.41pt}
\]
Now let $\mathcal P_m$ be the space of polynomials of degree at most $m$. The operator $\mathcal{L}$ preserves $\mathcal P_m$. More precisely, for $m\ge 0$,\vspace{-1.41pt}
\[
\mathcal{L}(u^m)
=
\int_0^1\bigl(c(cu)^m+(1-c)(c+(1-c)u)^m\bigr)\,\mathrm{d}F(c).\vspace{-1.41pt}
\]
Expanding the second term gives\vspace{-1.41pt}
\[
\mathcal{L}(u^m)
=
\lambda_{m}u^m+\text{a polynomial of degree at most }m-1.\vspace{-1.41pt}
\]
Therefore, in the canonical basis $\{1,x,x^2,\dots,x^m\}$, the restriction $\mathcal{L}|_{\mathcal P_m}$ is triangular, with diagonal entries\vspace{-1.41pt}
\[
\lambda_0,\lambda_1,\ldots,\lambda_{m}.\vspace{-1.41pt}
\]
In particular, $\lambda_1 = \mu_\ell$, defined in \eqref{eq - def mu ell}. Since these numbers are distinct, $\mathcal{L}|_{\mathcal P_m}$ is diagonalizable. Hence there exists, for each $m\ge 0$, a polynomial $\varphi_m$ of degree $m$ such that\vspace{-1.41pt}
\[
\mathcal{L}\varphi_m = \lambda_{m}\varphi_m.\vspace{-1.41pt}
\]
Let $p$ be a polynomial of degree $d$ satisfying $\int_0^1 p(u)\,\mathrm{d}u=0$. Then $p$ has a unique decomposition\vspace{-1.41pt}
\[
p=\sum_{m=1}^{d} a_m\varphi_m.\vspace{-1.41pt}
\]
Note that the term $m=0$ is absent because $p$ has zero mean. Indeed, we have\vspace{-1.41pt}
\begin{align*}
    \lambda_{m}\int_0^1 \!\varphi_m(u)\,\mathrm{d}u &= \int_0^1 \!\mathcal{L}\varphi_m(u)\,\mathrm{d}u = \int_0^1 \!\int_0^1\!
        \bigl(
            c \varphi_m(cu)
            +(1-c)\varphi_m(c+(1-c)u)
        \bigr)
        \,\mathrm{d}F(c)\,\mathrm{d}u \\
        &= \int_0^1 \biggl(\int_0^c\!
             \varphi_m(v) \,\mathrm{d}v
            +\int_c^{1} \!\varphi_m(v)\,\mathrm{d}v
        \biggr)\,\mathrm{d}F(c) = \int_0^1\! \int_0^{1}\! \varphi_m(v)\,\mathrm{d}v\,\mathrm{d}F(c) \\
        &=  \biggl(\int_0^1 \!\mathrm{d}F(c)\biggr) \biggl( \int_0^{1}\!\varphi_m(v)\,\mathrm{d}v \biggr) = \int_0^{1}\! \varphi_m(v)\,\mathrm{d}v .\vspace{-1.41pt}
\end{align*}
Therefore, for $m\ge 1$, we have $\lambda_{m}\ne 1$ which implies that $\smash{\int_0^1 \varphi_m(u)\,\mathrm{d}u=0}$. Thus\vspace{-1pt}
\[
\mathcal{L}^n p
=
\sum_{m=1}^{d} a_m\lambda_{m}^n\varphi_m.
\]
Since $\lambda_{m}<1$ for every $m\ge 1$, it follows that
\[
\|\mathcal{L}^n p\|_{L^1} \leq \sum_{m=1}^{d} |a_m|\lambda_{m}^n \|\varphi_m\|_{L_1} \leq \lambda_1^n \sum_{m=1}^{d} |a_m| \|\varphi_m\|_{L_1} = O(\lambda_1^n) \longrightarrow 0.
\]

Now let $\varepsilon>0$. Since polynomials are dense in $L^1([0,1])$, choose a polynomial $q$ such that
\[
\|\psi_0-q\|_{L^1}< \frac{\varepsilon}{2}.
\]
Set $p:=q-\int_0^1 q(u)\,\mathrm{d}u$. Since $\int_0^1 \psi_0(u)\,\mathrm{d}u=0$, we have
\[
\left|\int_0^1 q(u)\,\mathrm{d}u\right|
=
\left|\int_0^1(q(u)-\psi_0(u))\,\mathrm{d}u\right|
\le
\|q-\psi_0\|_{L^1}
<
\frac{\varepsilon}{2}.
\]
Therefore
\[
\|\psi_0-p\|_{L^1}
\le
\|\psi_0-q\|_{L^1}
+
\left|\int_0^1 q(u)\,\mathrm{d}u\right|
<
\varepsilon.
\]
Moreover, the operator $\mathcal{L}$ is an $L^1$-contraction. Indeed, for every $g\in L^1([0,1])$,
\begin{align*}
\|\mathcal{L}g\|_{L^1}
&\leq \int_0^1\int_0^1
\Big(c|g(ct)|+(1-c)|g(c+(1-c)t)|\Big)\,\mathrm{d}t\,\mathrm{d}F(c) \\
&= \int_0^1\left(\int_0^c |g(u)|\,\mathrm{d}u+\int_c^1 |g(u)|\,\mathrm{d}u\right)\mathrm{d}F(c)
= \|g\|_{L^1} .
\end{align*}
Hence $\|\mathcal{L}^n g\|_{L^1}\leq \|g\|_{L^1}$ for all $n\geq 0$. Therefore, we finally obtain
\[
\|\mathcal{L}^n \psi_0\|_{L^1}
\le
\|\mathcal{L}^n(\psi_0-p)\|_{L^1}
+
\|\mathcal{L}^n p\|_{L^1}
\le
\|\psi_0-p\|_{L^1}
+
\|\mathcal{L}^n p\|_{L^1} \leq \varepsilon + O(\lambda_1^n). \qedhere
\]
\end{proof}

\begin{rem}\label{rem - upperbound}
    The bound $\varepsilon+O(\mu_\ell^n)$ should not be interpreted as a genuine $\mu_\ell$-exponential convergence rate, since the implicit constant depend on $\varepsilon>0$ and may blow up as $\varepsilon\to0$. Nevertheless, we may show that such a rate holds under natural assumptions; for example, if $g_0=\sum_{n\geq0}a_n\varphi_n$ with $\sum_{n\geq0}|a_n|\|\varphi_n\|_{L^1}<\infty$, or if $g_0\in L^2([0,1])$ and the cuts are chosen uniformly.

    Moreover, $\lambda_1=\mu_\ell$ only gives an upper bound on the possible rate. Indeed, consider the root with PDF $g_0 := 1+\|\varphi_m\|^{-1}\varphi_m$, where $m\geq 2$ and $\|\cdot\|$ is the supremum norm. Then we have
    \[
    g_n-1 = \mathcal{L}^ng_0-1 =  \mathcal{L}^n (\varphi_m+1)-1 =  \mathcal{L}^n \varphi_m = \lambda_m^n \varphi_m.
    \]
    Consequently, $\|g_n-1\|_{L^1} = \lambda_m^n \|\varphi_m\|_{L^1} = O(\lambda_m^n)$, which has a faster convergence rate than $O(\mu_\ell^n)$. More generally, if $g_0$ is well approximated by a polynomial $p=1+\sum_{m=1}^{d}a_m\varphi_m$ with $a_1=\cdots=a_k=0$, then one may heuristically expect convergence at a rate closer to $\lambda_{k+1}$. In particular, when $g_0$ is symmetric and the cut distribution is symmetric, the coefficient $a_1$ should vanish, since $\varphi_1(t)=1-2t$ is antisymmetric. One may therefore expect a faster rate, closer to $\lambda_2$. This heuristic is explored numerically in further details in \Cref{sec - conv_unif}.
\end{rem}

Let us now return to the main problem and derive an analog of \Cref{thm - pdf} in the general case. If the random variables $\ell_n$ admit a PDF $h_n$, one would naturally seek to prove that
\[
\int_0^1 |h_n(t)-h(t)|\,\mathrm{d} t \longrightarrow 0,
\]
where $h(t)=t\bigl(f(t)+f(1-t)\bigr)$ is the density identified in \Cref{thm - pdf}. However, even when the normalized roots $r_n$ are absolutely continuous, the random variables $\ell^n$ need not admit densities.

\begin{ex}
    Let $r_0 \sim\mathrm{Beta}(2,2)$ and $c_0\sim \frac{1}{2}\delta_{1/3} + \frac{1}{2}\delta_{2/3}$. Then,
    \begin{equation*}
        \ell_1 \sim \frac{7}{27} \delta_{1/3} + \frac{20}{27} \delta_{2/3}.
    \end{equation*}
    Note that simultaneously, $r_1$, which is absolutely continuous, always has a polynomial density.
\end{ex}

In that setting, $L^1$ convergence of densities is no longer meaningful, so a different notion is needed. Although one could use total variation distance between the measures induced by $\ell_n$, which agrees with the above notion when densities exist, we avoid unnecessary machinery and instead use the simpler convergence of cumulative distribution functions in the supremum norm
\begin{equation*}
    \| f \| := \sup_{x\in[0,1]} |f(x)|,
\end{equation*}
which is sufficient for our purposes.

\begin{thm}\label{thm - pdf_general}
Let $F$ be the cumulative distribution function of the cuts $c_n$, $G_n$ the cumulative distribution function of the normalized roots $r_n$, $H(t)$ the cumulative distribution function defined in \Cref{thm - pdf}, and let $\ell_n$ be defined as in \eqref{def - l}. Then the cumulative distribution function of $\ell_n$ is
\[
H_n(t) = \int_0^t G_n(x) \, \mathrm{d} F(x) + \int_{1-t}^1 \bigl(1-G_n(x)\bigr)\,\mathrm{d} F(x).
\]
Moreover, if $r_0$ is absolutely continuous and $\mathbb{P}[0<c<1]>0$, then, for all $\varepsilon>0$,
\[
\|H_n(t)-H(t)\| \leq \varepsilon+O(\mu_\ell^n) \longrightarrow 0.
\]
\end{thm}
\begin{proof}
Proceeding as in \Cref{thm - pdf}, we find that the CDF of $\ell_n$ is
\begin{align*}
    H_n(t) &= \mathbb{P}[r_n\leq c_n\leq t]+\mathbb{P}[1-t\leq c_n < r_n] \\
    &= \int_0^t \mathbb{P}[r_n\leq x\,|\, c_n=x]\,\mathrm{d} F(x) + \int_{1-t}^1 \mathbb{P}[x\leq r_n\,|\, c_n=x] \,\mathrm{d} F(x) \\
    &= \int_0^t G_n(x) \, \mathrm{d} F(x) + \int_{1-t}^1 \bigl(1-G_n(x)\bigr)\,\mathrm{d} F(x).
\end{align*}
which establishes our first claim. Now, let $H(t)$ be the CDF defined in \Cref{thm - pdf}, namely\vspace{-.5pt}
\[
H(t)
=
\int_0^t x \,\mathrm{d} F(x)
+
\int_{1-t}^1 (1-x)\,\mathrm{d} F(x).
\]
\vspace{-6pt} Then we have\vspace{-1pt} 
\begin{align*}
    |H_n(t)-H(t)| &\leq \int_0^t \bigl| G_n(x)-x \bigr| \, \mathrm{d} F(x) + \int_{1-t}^1 \bigl|G_n(x)-x\bigr|\,\mathrm{d} F(x) \\
    &\leq \|G_n(t)-t\| \bigg( \int_0^t  \mathrm{d} F(x) + \int_{1-t}^1 \mathrm{d} F(x) \biggr) \\
    &\leq 2\,\|G_n(t)-t\|.
\end{align*}
Therefore, it follows from \Cref{thm - root_conv} that
\begin{align*}
    |H_n(t)-H(t)| &\leq 2\,\|G_n(t)-t\| = 2 \left| \int_0^t \bigl(g_n(t)-1\bigr) \,\mathrm{d}t \right| \\
    &\leq 2\int_0^t \bigl|g_n(t)-1\bigr| \,\mathrm{d}t \,\leq\, 2\,\|g_n-1\|_{L^1} \\
    &\leq \varepsilon+O(\mu_\ell^n)
\end{align*}
for all $\varepsilon>0$. Taking the supremum over $t\in[0,1]$ finally yields the desired conclusion.
\end{proof}

As a last corollary to this section, we finally conclude that the length $\ell_n$ will eventually shrink at each step by an average factor of approximately $1-2(\mu -\mu^2 - \sigma^2)$, as long as the initial root $r_0$ is absolutely continuous and the cuts $c$ satisfy $\mathbb{P}[0<c<1]>0$.

\begin{cor}\label{cor - mean_variance_gen}
    Let $r_n$ be the random normalized roots, as defined as in \eqref{eq - iteration_root}, let $c_n$ be the random cuts taken from a distribution with mean $\mu$ and variance $\sigma^2$, and let $\ell_n$ be defined as in \eqref{def - l}. If $r_0$ is absolutely continuous and $\mathbb{P}[0<c<1]>0$, then \vspace{-1pt}
    \begin{gather*}
        \mu_{\ell_n} \longrightarrow 1-2(\mu -\mu^2 - \sigma^2),\\
        \sigma_{\ell_n}^2 \longrightarrow 
        (\mu -\mu^2 - \sigma^2) \bigl(1-4 (\mu -\mu^2 - \sigma^2)\bigr).
    \end{gather*}
\end{cor}
\begin{proof}
    By \Cref{thm - pdf_general}, it follows that $\ell_n$ has the cumulative distribution function $H_n(t) = \int_0^t G_n(x) \, \mathrm{d} F(x) + \int_{1-t}^1 \bigl(1-G_n(x)\bigr)\,\mathrm{d} F(x).$ Moreover, if $H(t)=\int_0^t x \, \mathrm{d} F(x) + \int_{1-t}^1 \bigl(1-x\bigr)\,\mathrm{d} F(x)$, then $\|H_n-H\| \to 0$. It also follows from \Cref{cor - mean_variance_special} that
    \[
    \int_0^1 \!t \,\mathrm{d} H(t) = 1-2(\mu -\mu^2 - \sigma^2) = \mu_\ell \quad\text{and}\quad \int_0^1 \!t^2 \,\mathrm{d} H(t) = 1-3(\mu -\mu^2 - \sigma^2) = \mu_\ell^2+\sigma_\ell^2.
    \]
    Hence, it follows from integration by parts (the integration by parts formula for the Lebesgue--Stieltjes integral holds in this context  \cite[Theorem 14.1]{Saks_1937}) that
    \begin{align*}
        |\mu_{\ell_n}-\mu| &= \biggl|\int_0^1  \!\bigl(t\,\mathrm{d} H_n(t)-t\,\mathrm{d} H(t)\bigr) \biggr| =\biggl| H_n(1)-H_n(0)-H(1)+H(0)+\!\int_0^1  \!\bigl(H(t)-H_n(t)\bigr)\,\mathrm{d} t \biggr| \\
        &= \biggl|\int_0^1  \!\bigl(H(t)-H_n(t)\bigr)\,\mathrm{d} t \biggr| 
        \,\leq\,  \|H_n(t)-H(t)\| \,\longrightarrow\, 0 \\[-20.5pt]
    \end{align*}
    and \vspace{-6.5pt}
    \begin{align*}
        |\mu_{\ell_n}^2+\sigma_{\ell_n}^2-\mu_\ell^2-\sigma_\ell^2| &= \Biggl|\int_0^1 \bigl(t^2\,\mathrm{d} H_n(t)-t^2\,\mathrm{d} H(t)\bigr) \Biggr| \\
        &= \biggl| H_n(1)-H_n(0)-H(1)+H(0)+2\int_0^1  t\bigl(H(t)-H_n(t)\bigr)\,\mathrm{d} t \biggr| \\
        &= 2\,\biggl|\int_0^1  t\bigl(H_n(t)-H(t)\bigr)\,\mathrm{d} t \biggr| \,\leq\,  2\|H_n(t)-H(t)\| \,\longrightarrow\, 0.
    \end{align*}
    The conclusion then follows immediately from the triangle inequality.
\end{proof}

\section{Random multisection algorithm}\label{sec:4}

The classical bisection method can be generalized by taking $K$ equally spaced cuts in $[a,b]$ and keeping the subinterval where a sign change occurs, yielding a $(K+1)$-section algorithm. In that case, the interval shrinks by a factor $\frac{1}{K+1}$ at each step. We study the case where multiple uniformly random cuts are chosen over the interval. Let a vector of $K$ random cuts $\mathbf{c}_0 = \bigl(c_0^{1},\dots,c_0^{K}\bigr)$ be i.i.d.~with distribution $\mathcal{U}(0,1)$, and $r_0 \sim \mathcal{U}(0,1)$. We denote the order statistics by
\begin{equation*}
    \mathbf{c}_0^{\uparrow} = \bigl(c_0^{(1)},\dots,c_0^{(K)}\bigr),\qquad \text{ where } c_0^{(1)}\leq\dots\leq c_0^{(K)}
\end{equation*}
Moreover, let $c_0^{(0)}:=0$ and $c_0^{(K+1)}:=1$.
The rescaled root is then defined piecewise by the random variable \vspace{-1pt}
\begin{equation*}
    r_1 = T_{K+1}\big(\mathbf{c}_0,r_0\big)= 
        \frac{r_0-c_0^{(i)}}{c_0^{(i+1)}-c_0^{(i)}} \qquad\text{if }~ c_0^{(i)}\leq r_0<c_0^{(i+1)}\quad (i = 0,\dots,K),
\end{equation*}
which corresponds to a skewed $(K+1)$-adic map (\Cref{fig:kskeweddyadic}).

\begin{figure}[ht]
\begin{minipage}{.49\linewidth}
    \vspace{16pt}
\begin{figure}[H]
    \centering
    \begin{tikzpicture}[scale=.897]
    \definecolor{bluefill}{RGB}{25,25,112} 
    \definecolor{redfill}{rgb}{0.8, 0.25, 0.33} 
    \definecolor{greenfill}{RGB}{0,153,0}
    \definecolor{purplefill}{RGB}{255,148,0}
    
    \begin{axis}[
        axis lines = middle,
        x label style={at={(axis description cs:1.01,.1)},anchor=west},
        xmin = -0.09, xmax = 1.1,
        ymin = -0.09, ymax = 1.1,
        samples = 400,
        domain = 0:1,
        legend style={
        at={(0.1,0.99)},
        anchor=north west,
        draw=none
    }
    ]
    \addplot[bluefill, very thick] {x*((2*x^2-3*x+1)/(2^(0-1))+1)};
    \addlegendentry{$G_0$};
    
    \addplot[redfill, very thick,dashed] {x*((2*x^2-3*x+1)/(2^(1-1))+1)};
    \addlegendentry{$G_1$};
    
    \addplot[greenfill, very thick,dotted] {x*((2*x^2-3*x+1)/(2^(2-1))+1)};
    \addlegendentry{$G_2$};
    
    \addplot[purplefill,thick] {x*((2*x^2-3*x+1)/(2^(3-1))+1)};
    \addlegendentry{$G_3$};
    \end{axis}
    \end{tikzpicture}
    \vspace{-2pt}
    \caption{The CDF $G_0(t)=t(4t^2-6t+3)$, $G_1$, $G_2$ and $G_3$ with uniformly chosen cuts.}
    \label{fig:convergence}
\end{figure}
\end{minipage}%
\begin{minipage}{.02\linewidth}
\hfill
\end{minipage}%
\begin{minipage}{.49\linewidth}
\begin{figure}[H]
    \centering
    \hspace{-12pt}
    \begin{tikzpicture}[scale=.78]

    \definecolor{bluefill}{RGB}{25,25,112}
  \begin{axis}[
    width=8.5cm, height=7.58cm,
    xmin=0, xmax=1,
    ymin=0, ymax=1,
    xlabel={$r$},
    ylabel={$T_{4}(\mathbf{c},r)$},
    ylabel style={yshift=-8pt},
    xlabel style={yshift=5pt},
    axis lines=left,
    xtick={0, 0.3, 0.7, 0.8, 1},
    xticklabels = {$0$, $c^{(1)}$, $c^{(2)}$,$c^{(3)}$, $1$},
    ytick={0, 1},
    grid=major,
    grid style={dashed, gray!40},
  ]

  \addplot[bluefill, thick, domain=0:0.3, samples=2] {x/0.3};

  \addplot[bluefill, thick, domain=0.3:0.7, samples=2] {(x-0.3)/(0.7-0.3)};

  \addplot[bluefill, thick, domain=0.7:0.8, samples=2] {(x-0.7)/(0.8-0.7)};

  \addplot[bluefill, thick, domain=0.8:1, samples=2] {(x-0.8)/(1-0.8)};
  
  \end{axis}
\end{tikzpicture}
\vspace{-6pt}
    \caption{Skewed $(K+1)$-adic map ($K=3$).}
    \label{fig:kskeweddyadic}
\end{figure}
\end{minipage}
\end{figure}

As in \Cref{thm:stationarity}, the uniform distribution is also stationary for a $K$-cuts algorithm.

\begin{prop}
    If $r_0\sim\mathcal{U}(0,1)$, then $r_n\sim \mathcal{U}(0,1)$ for all $n\geq1$. That is, $\mathcal{U}(0,1)$ is a stationary distribution for the $K$-cuts algorithm. \label{thm:K-stationarity}
\end{prop}
\begin{proof}
The CDF of $r_1$ is represented by a $(K+1)$-dimensional integral 
\begin{align*}
    \mathbb{P}[T_{K+1}(\mathbf{c}_0,r_0) \leq t] &= \int_{[0,1]^{K+1}} \mathbf{1}_{L^-_t(T_{K+1})} \,\mathrm{d} r_0 \mathrm{d} c_0^{K}\cdots\mathrm{d} c_0^{1} \\
    &= K! \int_{ \Delta_K \times[0,1]} \mathbf{1}_{L^-_t(T_{K+1})} \mathrm{d} r_0 \mathrm{d} c_0^{(K)}\cdots\mathrm{d} c_0^{(1)}
\end{align*}
since the order statistics form a partition of the $K$-dimensional cube in $K!$ simplices and by symmetry of the uniform distribution, we can integrate over a single ascending simplex $\Delta_K$. We want the measure of the following geometric region:
\begin{align*}
    L_t^{-}(T_{K+1}) = \big\{(\mathbf{c}_0,r_0):  c_0^{(j)}\leq \,\,r_0 \leq t (c_0^{(j+1)}-c_0^{(j)}) + c_0^{(j)}, j = 0,\dots,K \big\}.
\end{align*}
The innermost integral is
\begin{align*}
    \int_0^1 \mathbf{1}_{L^-_t(T_{K+1})} \mathrm{d} r_0 = \sum_{j=0}^K \int_{c_0^{(j)}}^{t\bigl(c_0^{(j+1)}-c_0^{(j)}\bigr) + c_0^{(j)}} \mathrm{d} r_0 = \sum_{j=0}^K t\bigl(c_0^{(j+1)}-c_0^{(j)}\bigr) = t
\end{align*}
because of the telescopic summation. This leaves only $t K! \int_{\Delta_K} \mathrm{d} c_0^{(K)}\dots\mathrm{d} c_0^{(1)}$, where the integral is the probability that $K$ uniformly chosen random numbers in $[0,1]$ appear in nondecreasing order. Since the $K!$ possible orderings are equiprobable and only one is nondecreasing, up to probability-zero ties, the integral equals $1/K!$. Thus,
\begin{align*}
    t K! \int_{\Delta_K} \mathrm{d} c_0^{(K)}\cdots\mathrm{d} c_0^{(1)} = t K!\frac{1}{K!} = t.
\end{align*} 
Thus, $\mathbb{P}[T_{K+1}(\mathbf{c}_0,r_0) \leq t] = t$ and $r_1\sim \mathcal{U}(0,1)$. By induction, the same holds for any $r_n$ with $n>1$. Hence, the uniform distribution of the root is also stationary for a $K$-cuts algorithm.
\end{proof}

Let us now determine the expected length of $\ell_1$.

\begin{prop}\label{cor - K-mean}
    If $r_0 \sim \mathcal{U}(0,1)$ and $\mathbf{c}_0=(c_0^1,\dots,c_0^K)$ are i.i.d. with distribution $\mathcal{U}(0,1)$, then\vspace{-1pt}
    \begin{gather*}
        \mathbb{E}[\ell_1] = \frac{2}{K+2}.
    \end{gather*}
\end{prop}
\begin{proof}
As in the previous sections, the scaling factor is defined piecewise by
\begin{equation*}
    \ell_1 = 
        c_0^{(j+1)}-c_0^{(j)} \qquad\text{if }~ c_0^{(j)}\leq r_0<c_0^{(j+1)}\quad (j = 0,\dots,K).
\end{equation*}
Similarly to \eqref{eq:El1r0}, we can write\vspace{-2pt}
\begin{align*}
    \mathbb{E}[\ell_1\,|\, r_0] = \sum_{j=0}^K \mathbb{E}\bigl[c_0^{(j+1)}-c_0^{(j)} \,|\, c_0^{(j)}< r_0 \leq c_0^{(j+1)}\bigr] \cdot \mathbb{P}\bigl[c_0^{(j)}< r_0\leq c_0^{(j+1)}\bigr].\\[-19pt]
\end{align*}
First note that we have $c_0^{(j)}\leq r_0 < c_0^{(j+1)}$ if and only if $j$ of the random cuts fall in $[0,r_0]$ and $K-j$ fall in $(r_0,1]$. It is a standard fact that the probability of exactly $j$ points fall in $[0,r_0]$ (and the remaining $K-j$ in $[r_0,1]$) is the binomial probability\vspace{-2pt}
\begin{align*}
    \mathbb{P}\bigl[c_0^{(j)} < r_0 \leq c_0^{(j+1)}\bigr] = \binom{K}{j}r_0^j(1-r_0)^{K-j}.\\[-19pt]
\end{align*}
Now, we seek to determine\vspace{-2pt}
\begin{equation*}
    \mathbb{E}\bigl[c_0^{(j+1)}-c_0^{(j)} \,|\, c_0^{(j)} < r_0 \leq c_0^{(j+1)}\bigr].
\end{equation*}
Since we assume that $c_0^{(j)} < r_0 \leq c_0^{(j+1)}$, we know that $j$ of the random variables fell in $[0,r_0]$, and the remaining $K-j$ in $[r_0,1]$. Hence, if we denote by $X$ the maximum of the $j$ points in $[0,r_0]$ and $Y$ the minimum of the $K-j$ points in $[r_0,1]$, we find that
\[
\mathbb{E}\bigl[c_0^{(j+1)}-c_0^{(j)} \,|\, c_0^{(j)} < r_0 \leq c_0^{(j+1)}\bigr] = \mathbb{E}[Y-X].
\]
Moreover, we may write
\[
\mathbb{E}[Y-X] = \mathbb{E}\bigl[(Y-r_0)-X\bigr]+r_0 =: \mathbb{E}[Z-X]+r_0,
\]
where $X=\max\{c_0^{(i)}:i=0,1,\dots,j\}$, $Z=Y-r_0 = \min\{c_0^{(i)}-r_0:i=j+1,\dots,K+1\}$ and the $c_0^{(i)}$ and $c_0^{(i)}-r_0$ are uniformly distributed on $[0,r_0]$ and $[0,1-r_0]$, respectively. Hence, $X$ is the maximum of $j$ uniformly random variables on $[0,r_0]$ and $Z$ is the minimum of $K-j$ uniformly random variables on $[0,1-r_0]$.  
Therefore, these quantities are precisely the first order statistic of a sample of $n$ uniform variables on $(0,1)$. It is known (see \cite[Example 3.1.1]{MR1994955}) that if $X_j\sim \mathcal{U}(0,s)$, then they follow a $\text{Beta}(1,n)$ (resp.~$\text{Beta}(n,1)$) random variable and their means are
\[
\mathbb{E}\bigl[\max\{X_1,X_2,\dots,X_n\}\bigr] = \frac{ns}{n+1} \qquad\text{and}\qquad \mathbb{E}\bigl[\min\{X_1,X_2,\dots,X_n\}\bigr] = \frac{s}{n+1}.
\]
Hence, it follows that
\[
\mathbb{E}\bigl[c_0^{(j+1)}-c_0^{(j)} \,|\, c_0^{(j)}< r_0 \leq c_0^{(j+1)}\bigr] = \mathbb{E}[Z-X]+r_0 = \frac{1-r_0}{K-j+1} - \frac{jr_0}{j+1} + r_0 = \frac{1-r_0}{K-j+1} + \frac{r_0}{j+1}.
\]
Consequently, we finally find
\begin{align*}
    \mathbb{E}[\ell_1\,|\, r_0] &= \sum_{j=0}^K \mathbb{E}\bigl[c_0^{(j+1)}-c_0^{(j)} \,|\, c_0^{(j)}< r_0 \leq c_0^{(j+1)}\bigr] \cdot \mathbb{P}\bigl[c_0^{(j)} < r_0 \leq c_0^{(j+1)} \bigr] \\ 
    &= \sum_{j=0}^K \left( \frac{1-r_0}{K-j+1} + \frac{r_0}{j+1} \right) \cdot\binom{K}{j}r_0^j(1-r_0)^{K-j}.
\end{align*}
For the first sum, use the identity $\frac{1}{j+1}\binom{K}{j} = \frac{1}{K+1}\binom{K+1}{j+1}$ and set $i=j+1$ to obtain
\[
\sum_{j=0}^K \frac{r_0}{j+1} \binom{K}{j} r_0^j (1-r_0)^{K-j} 
= \frac{1}{K+1} \sum_{i=1}^{K+1} \binom{K+1}{i} r_0^i (1-r_0)^{K+1-i} 
= \frac{1-(1-r_0)^{K+1}}{K+1}.
\]
For the second sum, set $i=K-j$ and apply the same argument:
\[
\sum_{j=0}^K \frac{1-r_0}{K-j+1} \binom{K}{j} r_0^j (1-r_0)^{K-j} = \frac{1-r_0^{K+1}}{K+1}.
\]
Combining both yields
\[
\mathbb{E}[\ell_1\,|\, r_0]  = \sum_{j=0}^K \left( \frac{r_0}{j+1} + \frac{1-r_0}{K-j+1} \right)\! \binom{K}{j} r_0^j (1-r_0)^{K-j} 
= \frac{2 - r_0^{K+1} - (1-r_0)^{K+1}}{K+1}.
\]
As a final step, integrate over $r_0$ to get the expectation of $\ell_1$, since $r_0\sim \mathcal{U}(0,1)$. We then obtain
\begin{align*}
    \mathbb{E}[\ell_1] &= \mathbb{E}\bigl[\mathbb{E}[\ell_1\,|\, r_0]\bigr] = \int_0^1 \frac{2 - r_0^{K+1} - (1-r_0)^{K+1}}{K+1}\,\mathrm{d} r_0 \\
    &=  \frac{2}{K+1} - \int_0^1 \frac{r_0^{K+1} + (1-r_0)^{K+1}}{K+1}\,\mathrm{d} r_0 \\
    %
    %
    &= \frac{2}{K+1} - \frac{2}{(K+1)(K+2)} = \frac{2}{K+2}. \qedhere
\end{align*}
\end{proof}

As in the single-cut case, the following can be deduced inductively from the above.
\begin{cor}
    Assume $r\sim\mathcal{U}(0,1)$. The $(K+1)$-section algorithm obtained by taking a vector of $K$ independent and uniformly distributed cuts in $[0,1]$ at each iteration satisfies
    \begin{equation*}
        \mathbb{E}[\ell_n] = \frac{2}{K + 2}\quad\text{for all $n>0$.}
    \end{equation*}
\end{cor}
Hence, assuming that the root $r_0$ is distributed uniformly on $[0,1]$, the stochastic bisection method with $K$ cuts uniformly chosen has a reduction factor which is approximately twice as large as the deterministic $(K+1)$-section method. Interestingly enough, a random trisection method ($K=2$) with uniform cuts has, on average, the same scaling factor at each step as the deterministic bisection method, assuming the root is uniformly distributed.

\section{Numerical Experiments}\label{sec:5}

We perform numerical experiments to validate the previous theoretical results. The values below are obtained from the original bisection method with a probabilistic cut (\Cref{alg:bisection}), rather than from the scale-invariant version, to confirm that our results remain valid in this setting.

\subsection{Expected Convergence Rate}

To empirically estimate the expected scaling factor, consider the following experiment:
\begin{enumerate}
    \item Generate $r \sim \mathcal{U}(0,1)$;
    \item Run \Cref{alg:bisection} with cuts $c_n\sim \mathcal{D}(0,1)$, for $f(x) = x - r$, with a very low tolerance (e.g., $10^{-15}$) and a large number of iterations (e.g., $N=30$);
    \item Compute the scaling factors for all iterations, 
    $ \frac{b_{i+1}-a_{i+1}}{b_i - a_i}$, and store them; 
    \item Repeat steps 1 to 3 a large number of times (e.g., 500 runs);
    \item Compute a 95\% confidence interval (CI) for the mean using bootstrap.
\end{enumerate}

\begin{table}[ht]
    \centering
    \begin{tabular}{ccccc}
        \toprule
        $\mathcal{D}$ & $\mathcal{U}(0,1)$ & $\mathrm{Beta}(2,2)$ & $\mathrm{Beta}(0.5,2)$ & $\mathrm{Bates}(20)$ \\
        \midrule
        $\mathbb{E}[\ell_i]$ & $[0.654, 0.677]$ & $[0.595, 0.602]$ & $[0.767, 0.775]$ & $[0.507, 0.509]$\\
        \midrule
        $(\mathbb{E}[L_N])^{\frac{1}{N}}$ & $[0.655, 0.676]$ & $[0.599, 0.612]$ & $[0.764, 0.781]$ & $[0.508, 0.510]$
        \\
        \midrule
        Theoretical & $0.\overline{6}$ & 0.6 & $0.7714\dots$ &$0.508\overline{3}$\\
        \bottomrule
    \end{tabular}
    \caption{Expected scaling factor for multiple distributions.}
    \label{tab:scalingfactors}
\end{table}

The main advantage of using bootstrap confidence intervals is that they are non-parametric, meaning that they are expected to work without any knowledge (or mild hypotheses) about the underlying distribution \cite{Efron_Tibshirani_1994}. The results are shown in \Cref{tab:scalingfactors}. For all of the tested distributions, the theoretical value lies within the confidence intervals and the confidence intervals are very small, providing statistical evidence that the theoretical scaling factors are indeed correct. Moreover, a distribution with a lot of mass at $1/2$ and a small variance like $\mathrm{Bates}(20)$ has a convergence rate very close to the classical bisection algorithm. The length of the last interval for each run is also measured to verify that it also satisfies the conclusion of \Cref{cor - final}.

We repeat the same experiment for $K$-cuts variants of the algorithm (\Cref{tab:scalingfactorsK}), where the distribution of the root and cuts are i.i.d. $\mathcal{U}(0,1)$. The expected scaling factors are also matching the theoretical expression as a function of the number of cuts.

\begin{table}[ht]
    \centering
    \begin{tabular}{ccccc}
        \toprule
        $K$ & $2$ & $3$ & $4$\\
        \midrule
        $\mathbb{E}[\ell_i]$ & $[0.495, 0.504]$ & $[0.394, 0.402]$ & $[0.331, 0.338]$\\
        \midrule
        $(\mathbb{E}[L_N])^{\frac{1}{N}}$ & $[0.489, 0.509]$ & $[0.385, 0.401]$ & $[0.331, 0.360]$
        \\
        \midrule
        Theoretical & $0.5$ & 0.4 & $0.\overline{3}$\\
        \bottomrule
    \end{tabular}
    \caption{Expected scaling factor for random $(K+1)$-section.}
    \label{tab:scalingfactorsK}
\end{table}

The last experiment of this section considers fixed root values: one near the boundary ($r=0.1$), one near the center ($r=0.45$), and one in between ($r=0.3$). To reach a tolerance of $10^{-8}$, the deterministic bisection method requires 27 iterations, independently of the root's position. Over 1000 runs for each initial root, we compute a 95\% bootstrap CI for the average number of iterations needed to reach this tolerance (\Cref{tab:nbriter}). We also compute the range, i.e., the minimal and maximal number of iterations needed, and the number of runs performing at least as well as deterministic bisection, which we call \textit{lucky runs}. We estimate $\mathbb{P}[N\leq 27]$ by a proportion with Wilson CIs \cite{Franco_Little_Louis_Slud_2019}. Unsurprisingly, $\mathrm{Beta}(2,0.5)$ performs worst, since it is biased away from the subinterval initially containing the root. $\mathrm{Beta}(0.5,2)$ has the opposite problem: if the root eventually becomes centered in the working interval, the cuts remain biased to the left and become suboptimal. Still, lucky runs occur with very small probability, below $1\%$. The uniform distribution also performs poorly, requiring on average 10 more iterations and yielding only a small probability of lucky runs in all cases. $\mathrm{Beta}(2,2)$ requires about 5 more iterations on average, but has a significantly higher probability of lucky runs, which can save up to 10 iterations. Finally, $\mathrm{Bates}(20)$ has a $50\%$ to $60\%$ chance of producing a lucky run; however, since its variance is small, lucky runs improve performance only modestly, saving at most 2 or 3 iterations. Moreover, on average, it requires at most one more iteration than deterministic bisection to converge. The effect of the root position is consistent with \Cref{fig:explength} and \Cref{fig:explength_beta22} for the $\mathcal{U}(0,1)$ and $\mathrm{Beta}(2,2)$ cases: convergence seems slightly faster near the boundary and slowest near the center.

\begin{table}[ht]
    \centering
    \begin{tabular}{ccccccc}
        \toprule
        $r$ & &$\mathcal{U}(0,1)$&$\mathrm{Beta}(2,2)$ & $\mathrm{Beta}(0.5,2)$ & $\mathrm{Beta}(2,0.5)$ & $\mathrm{Bates}(20)$ \\
        \midrule
        \multirow{3}*{0.1}\!\!  & Mean& $[37.0,37.8]$ & $[31.8,32.3]$ & $[49.4,50.6]$ & $[54.0,55.4]$ & $[27.4,27.5]$\\
        & Range&$[19,56]$ & $[17,47]$ & $[17,89]$ & $[19,88]$ & $[25,31]$
        \\
        & $\mathbb{P}[N\leq 27]$\!\!& $[0.03,0.06]$ & $[0.12,0.16]$ & $[0.00, 0.01]$  & $[0.00,0.01]$ & $[0.53,0.59]$
        \\
        \midrule
        \multirow{3}*{0.3}\!\! & Mean& $[38.0,38.7]$ & $[32.4,32.9]$ & $[51.7,53.0]$ & $[53.7,54.4]$ & $[27.4,27.5]$\\
        & Range &$[22,56]$& $[21,46]$& $[17,95]$& $[23,93]$& $[24,31]$\\
        & $\mathbb{P}[N\leq 27]$\!\! & $[0.02,0.04]$ & $[0.08,0.12]$ & $[0.00,0.02]$ & $[0.00,0.01]$ & $[0.52,0.58]$\\
        \midrule
        \multirow{3}*{0.45}\!\! & Mean& $[37.9,38.7]$ & $[32.2,32.7]$ & $[52.6,53.8]$ & $[53.9,54.0]$ & $[27.4,27.5]$
        \\
        & Range &$[20,62]$& $[20,47]$& $[21,95]$& $[24,89]$& $[25,31]$\\
        & $\mathbb{P}[N\leq 27]$\!\! &$[0.03,0.05]$& $[0.08,0.11]$& $[0.00,0.01]$ & $[0.00,0.01]$ & $[0.49,0.55]$\\
        \bottomrule
    \end{tabular}
    \caption{Statistics on the number of iterations for convergence with a fixed initial root.}
    \label{tab:nbriter}
\end{table}
\vspace{-6pt}

\subsection{Convergence to the Uniform Distribution}\label{sec - conv_unif}

To verify convergence toward the stationary uniform distribution for different initial root distributions $\mathcal{R}$, we consider the following experiment:
\begin{enumerate}
    \item Generate $r \sim \mathcal{R}(0,1)$;
    \item Run \Cref{alg:bisection} with uniformly distributed cuts, for $f(x) = x - r$, with a very low tolerance (e.g., $10^{-15}$) and a large number of iterations (e.g., $N=40$);
    \item Normalize the root to its relative position in $[0,1]$ with $r_N \gets \frac{r - a_N}{b_N-a_N}$ and store it;
    \item Repeat steps 1 to 3 a large number of times (e.g., 1000 runs).
\end{enumerate}

Using the stored normalized roots, we can then validate that their distribution is nearly uniform with a Q-Q plot. We use three initial distributions $\mathcal{R}$:
\begin{itemize}
    \item The uniform distribution;
    \item A symmetric distribution, $\mathrm{Beta}(2,2)$;
    \item An asymmetric distribution, $\mathrm{Beta}(0.5, 2)$.
\end{itemize}
\Cref{fig:qq_unif}, \Cref{fig:qq_beta22} and \Cref{fig:qq_beta052} show that for the three distributions considered, the quantiles of the generated data follow closely the quantiles of the desired uniform distribution. This confirms that after a large number of iterations, even for a strongly asymmetric distribution, the distribution of the rescaled root approaches the uniform distribution.

\begin{figure}[ht]
\begin{minipage}{.49\linewidth}
    \begin{figure}[H]

    \vspace{-2pt}
    \caption{Q-Q plot for $r\sim \mathcal{U}(0,1)$.}
    \label{fig:qq_unif}
\end{figure}
\end{minipage}%
\begin{minipage}{.02\linewidth}
\hfill
\end{minipage}%
\begin{minipage}{.49\linewidth}
    \begin{figure}[H]
    %
    \vspace{-2pt}
    \caption{Q-Q plot for $r\sim \mathrm{Beta}(2,2)$.}
    \label{fig:qq_beta22}
\end{figure}
\end{minipage}
\end{figure}

Furthermore, we can approximate numerically the convergence rate of some initial distributions to the uniform distribution in the $L^1$ norm. The $L^1$ norm is estimated numerically using the area of an histogram with 100 bins. The following simulation is used to estimate the convergence rate:
\begin{enumerate}
    \item Draw $M=10,\!000,\!000$ initial roots from the distribution $\mathcal{R}$;
    \item Do a single step of \Cref{alg:bisection} with $c_n\sim \mathcal{U}(0,1)$ for each of the $M$ roots. Rescale each root with respect to $[a_n,b_n]$ and store its value;
    \item Compute the $L^1$ norm with the previously computed rescaled roots and store its value;
    \item Overwrite the roots with their rescaled counterparts;
    \item Repeat steps 2 to 4 until reaching a large enough number of iterations (e.g., $N=9$);
    \item Using least squares regression, fit a linear model to the log of the stored values of the $L^1$ distance.
\end{enumerate}
As it can be seen on \Cref{fig:l1_beta35}, because of the random nature of the experiment, the convergence cannot be fully observed beyond the noise level. Therefore, the algorithm are only done up to $N=9$ iterations to avoid this regime where the convergence stagnates.

The different initial distributions which will be tested are characterized by their properties in (\Cref{tab:distprop}). We define the PDF $p(x) = \frac{1-\alpha}{2^\alpha}|x-1/2|^{-\alpha}$ with $\alpha = 0.9$.
\begin{table}[]
    \centering
    \begin{tabular}{ccccc}
        \toprule
        Distribution & Polynomial & Bounded & Symmetric & Highly-concentrated\\
        \midrule
        Bates(20) &  &\checkmark & \checkmark & \checkmark \\
        \midrule
         Beta(2,2) & \checkmark & \checkmark & \checkmark & \\
        \midrule
         $p$ &   && \checkmark & \checkmark\\
         \midrule
         Beta(3,5) & \checkmark & \checkmark & & \\ 
         \midrule
         Beta(1,20) & \checkmark & \checkmark & & \checkmark\\
         \midrule
         Beta(0.1,2) && & &\checkmark \\
        \bottomrule
    \end{tabular}
    \caption{Properties of the different distributions studied.}
    \label{tab:distprop}
\end{table}

The results for the convergence rates of the different initial densities are presented in \Cref{tab:convrate}. We observe that the symmetric distributions ($\mathrm{Bates}(20)$, $\mathrm{Beta}(2,2)$ and $p$) converges to the uniform distribution at a rate close to $1/2$, which is expected following \Cref{rem - upperbound} ($\lambda_2=1/2$ for the eigen-polynomials in the uniform measure), even though $p$ and $\mathrm{Bates}(20)$ have small dispersion and a large concentration of mass at $1/2$, with $p$ even being unbounded. We note that the singularity of $p$ appears to not be a problem, since although its PDF is not well approximated by polynomials, the theoretical PDF of $r_1$ is bounded and the one of $r_2$ is bounded and $C^2$. Hence, we may apply the convergence analysis to $r_2$, which is well approximated by polynomials, and thus get back the predicted convergence rate.

By contrast, the asymmetric distribution Beta(1,20) converge more slowly, at a rate closer to the bound given by the theorem. For the distribution Beta(0.1,2), the convergence rates is much larger than the rate given by the theorem. However, as mentioned in \Cref{rem - upperbound}, the convergence also has a strong dependence on the approximability of the initial density by polynomial eigenfunctions of the operator $\mathcal{L}$. Since Beta(0.1,2) is poorly approximated by polynomials, this heuristically explain the observed convergence rate.

However, we can be much more precise than this (although we stay voluntarily vague for this explanation). Indeed, by \eqref{eq - PDF}, the PDF of the root $r_{1}$ is given by 
\begin{align*}
    g_{1}(u) &= \int_0^1
        \bigl(
            c g_0(cu)
            +(1-c)g_0(c+(1-c)u)
        \bigr)
        \,\mathrm{d}c \\
        &= \frac{1}{u^2}\int_0^u c g_0(c)\,\mathrm{d}c
            + \frac{1}{(1-u)^2}\int_u^1 (1-c) g_0(c)\,\mathrm{d}c
\end{align*}
If $g_0$ can be well approximated by a polynomial, then the previous approach holds with good precision. However, near 0, the distribution Beta(0.1,2) is better approximated by the function $A/x^{0.9}$, for some constant $A>0$. In that case, we have
\begin{align*}
    g_{1}(u) &= \frac{1}{u^2}\int_0^u c g_0(c)\,\mathrm{d}c
            + \frac{1}{(1-u)^2}\int_u^1 (1-c) g_0(c)\,\mathrm{d}c \\
            &\approx \frac{A}{u^2}\int_0^u c^{0.1} \,\mathrm{d}c = \frac{10A}{11} u^{-0.9} = \frac{10}{11} g_0(u).
\end{align*}
Proceeding inductively, we find that the singular part of Beta(0.1,2) decays like $(10/11)^n = 0.\overline{90}^n$. Since this singular part is what influences the $L^1$ distance to the uniform distribution's PDF the most, one might expect that the convergence rate behaves like $0.\overline{90}^n$, which is precisely what is observed numerically.

These results suggest that polynomial approximability is the main bottleneck for the convergence rate. Hence, once sufficient regularity is assumed, symmetry becomes the dominant factor, often bringing the rate down to nearly $1/2$.

\begin{figure}[ht]
\begin{minipage}[t]{.49\linewidth}
    \begin{figure}[H]

    \vspace{-2pt}
    \caption{Q-Q plot for $r\sim \mathrm{Beta}(0.5,2)$.}
    \label{fig:qq_beta052}
\end{figure}
\end{minipage}%
\begin{minipage}{.02\linewidth}
\hfill
\end{minipage}%
\begin{minipage}[t]{.49\linewidth}
    \begin{figure}[H]
    %

    \vspace{2pt}
    \caption{Convergence of the $L^1$ norm for $r\sim \mathrm{Beta}(3,5)$ (red points) and fitted exponential model (blue line) on a semilogarithmic scale.}
    \label{fig:l1_beta35}
\end{figure}
\end{minipage}
\end{figure}

\begin{table}[ht]
    \centering
    {\footnotesize
    \begin{tabular}{cccccccc}
        \toprule
         $\mathcal{R}$ & $\mathrm{Bates}(20)$ & $\mathrm{Beta}(2,2)$ & $p$ & $\mathrm{Beta}(3,5)$ & $\mathrm{Beta}(1,20)$ & $\mathrm{Beta}(0.1,2)$ \\
        \midrule
        $\| g_n -1 \|_1$ & 0.519 & 0.546 & 0.533 & 0.613 & 0.686 & 0.913\\
        \midrule
        $|\mu_{\ell_n}-\mu_\ell|$ & 0.424 & 0.504 & 0.433 & 0.489 & 0.578 & 0.915
        \\
        \midrule
        Theoretical & $0.5$ & $0.5$ & $0.5$ & $0.\overline{6}$ &$0.\overline{6}$&$0.\overline{6}$\\
        \bottomrule
    \end{tabular}
    }
    \caption{Approximate convergence rate for multiple initial root distributions.}
    \label{tab:convrate}
\end{table}

The distance between the means $\mu_{\ell_n}$ and $\mu_\ell$ should follow a similar decay rate. To verify this, $\ell_i$ is computed at each iteration of the algorithm and averaged over a large number of runs, with the initial $r$ drawn with respect to the distribution $\mathcal{R}$. The cut distribution is always $\mathcal{U}(0,1)$, so $\mu_{\ell_n}\to 2/3$. In \Cref{tab:convrate}, we note that the means converge at a rate similar or faster to that of the density $g_n$.

\subsection{Pathological case: correlation in the \texorpdfstring{$\ell_i$}{ℓᵢ}}

The uniformity of $r$ was instrumental in proving \Cref{thm - ind}, namely that $\ell_i$ is independent of $\ell_j$ for $i\neq j$. We now construct a counterexample when the root is not uniformly distributed. Consider a distribution with most of its mass near one boundary, for instance $\mathrm{Beta}(5,50)$, so that the root is near the left end of the interval with high probability. If the cuts follow the same distribution, they are also close to the root with high probability. The following phenomenon then occurs: if the cut falls slightly to the left of the root, $\ell_1$ is very long (\Cref{fig:c0leqr0}); if it falls slightly to the right, $\ell_1$ is shorter (\Cref{fig:c0geqr0}). On the next iteration, after rescaling, the root is either near the left boundary if $\ell_1$ was large, or near the right boundary if $\ell_1$ was small.

\begin{figure}[ht]
\begin{minipage}{.49\linewidth}
    \begin{figure}[H]
    \centering
    \begin{tikzpicture}[scale=5.3]
\draw[|-|] (0,0) -- (1.1,0);
\draw (0,0) node[below,yshift=-3pt] {$0$};
\draw (1.1,0) node[below,yshift=-3pt] {$1$};
\draw[dashed,semithick,decorate, decoration={aspect=0, segment length=2pt, amplitude=0.5pt}]
  plot[smooth] coordinates {(0,0) (0.03,0.1) (0.1,0.4) (0.2,0.08) (0.3,0.02) (0.5,0.01) (0.7,0.007) (1.1,0.005)};
\fill (0.08,0) circle (0.01) node[below] {$c_0$};
\fill (0.14,0) circle (0.01) node[below] {$r_0$};
\draw[|<->|, semithick] (0.08,-0.12) -- (1.1,-0.12) node[below,midway] {$\ell_1$};
\end{tikzpicture}
\end{figure}
\end{minipage}%
\begin{minipage}{.02\linewidth}
\hfill
\end{minipage}%
\begin{minipage}{.49\linewidth}
    \begin{figure}[H]
    \centering
    \begin{tikzpicture}[scale=5.3]
\draw[|-|] (0,0) -- (1.1,0);
\draw (0,0) node[below,yshift=-3pt] {$0$};
\draw (1.1,0) node[below,yshift=-3pt] {$1$};
\draw[dashed,semithick,decorate, decoration={aspect=0, segment length=2pt, amplitude=0.5pt}]
  plot[smooth] coordinates {(0,0) (0.03,0.1) (0.1,0.4) (0.2,0.08) (0.3,0.02) (0.5,0.01) (0.7,0.007) (1.1,0.005)};
\fill (0.12,0) circle (0.01) node[below] {$c_1$};
\fill (0.05,0) circle (0.01) node[below] {$r_1$};
\draw[|<->|, semithick] (0,-0.12) -- (0.12,-0.12) node[below,midway] {$\ell_2$};
\end{tikzpicture}
\end{figure}
\end{minipage}
\vspace{-3pt}
\caption{First and second iterations when $c_0<r_0$. In dashed, the PDF of $r_0$, $c_0$ and $c_1$.}
    \label{fig:c0leqr0}
\end{figure}
\vspace{-2pt}

In the first case, the new rescaled root will be pushed further to the left, reducing the probability of picking a cut to its left, making it more likely to pick a cut to its right, and having a small $\ell_2$. In the second case, the next cut will be near the left edge with high probability, and $\ell_2$ will likely be large. Hence, there is an inverse relationship between $\ell_1$ and $\ell_2$ (if one was large, the other should likely be small).

\begin{figure}[ht]
\begin{minipage}{.49\linewidth}
    \begin{figure}[H]
    \centering
    \begin{tikzpicture}[scale=5.3]
\draw[|-|] (0,0) -- (1.1,0);
\draw (0,0) node[below,yshift=-3pt] {$0$};
\draw (1.1,0) node[below,yshift=-3pt] {$1$};
\draw[dashed,semithick,decorate, decoration={aspect=0, segment length=2pt, amplitude=0.5pt}]
  plot[smooth] coordinates {(0,0) (0.03,0.1) (0.1,0.4) (0.2,0.08) (0.3,0.02) (0.5,0.01) (0.7,0.007) (1.1,0.005)};
\fill (0.09,0) circle (0.01) node[below] {$r_0$};
\fill (0.14,0) circle (0.01) node[below] {$c_0$};
\draw[|<->|, semithick] (0,-0.12) -- (0.14,-0.12) node[below,midway] {$\ell_1$};
\end{tikzpicture}
\end{figure}
\end{minipage}%
\begin{minipage}{.02\linewidth}
\hfill
\end{minipage}%
\begin{minipage}{.49\linewidth}
    \begin{figure}[H]
    \centering
    \begin{tikzpicture}[scale=5.3]
\draw[|-|] (0,0) -- (1.1,0);
\draw (0,0) node[below,yshift=-3pt] {$0$};
\draw (1.1,0) node[below,yshift=-3pt] {$1$};
\draw[dashed,semithick,decorate, decoration={aspect=0, segment length=2pt, amplitude=0.5pt}]
  plot[smooth] coordinates {(0,0) (0.03,0.1) (0.1,0.4) (0.2,0.08) (0.3,0.02) (0.5,0.01) (0.7,0.007) (1.1,0.005)};
\fill (0.12,0) circle (0.01) node[below] {$c_1$};
\fill (0.8,0) circle (0.01) node[below] {$r_1$};
\draw[|<->|, semithick] (0.12,-0.12) -- (1.1,-0.12) node[below,midway] {$\ell_2$};
\end{tikzpicture}
\end{figure}
\end{minipage}
\vspace{-3pt}
\caption{First and second iterations when $c_0>r_0$. In dashed, the PDF of $r_0$, $c_0$ and $c_1$.}
    \label{fig:c0geqr0}
\end{figure}
\vspace{-2pt}

To demonstrate that this phenomenon is happening in practice, we run the bisection algorithm using $\mathrm{Beta}(5,50)$ for the initial root and the cuts, for $M=10000$ different initial roots. We let the algorithm complete 14 iterations, we collect vectors $\big(\ell_1^{(m)},\dots,\ell_{14}^{(m)}\big)$ ($m=1,\dots,M$), and then assemble the correlation matrix to study the dependence structure between the $\ell_i$ (\Cref{fig:corrbeta}). We repeat the same experiment for uniform initial root distribution (\Cref{fig:corrunif}). We see that in the uniform case, all the variables seem to be uncorrelated, which is consistent with the independence hypothesis. However, when $r_0\sim\mathrm{Beta}(5,50)$, there is a strong negative correlation between $\ell_1$ and $\ell_2$, as it was predicted. Moreover, the dependence on $\ell_1$ and $\ell_2$ seems to carry over as far as $\ell_{10}$, before vanishing. Because the stationary distribution is uniform, the variables appearing later in the process have likely very weak dependence.
\vspace{-6pt}

\begin{figure}[H]
\begin{minipage}{.02\linewidth}
\hfill
\end{minipage}%
\begin{minipage}[t]{.46\linewidth}
    \begin{figure}[H]
    \centering
    \begin{tikzpicture}[scale=0.8,xscale=.85]

\definecolor{darkgray176}{RGB}{176,176,176}

\definecolor{bluefill}{RGB}{25,25,112} 
    \definecolor{redfill}{rgb}{0.8, 0.25, 0.33} 

\begin{axis}[
very thick,
colorbar,
colorbar style={ylabel={}},
colormap={mymap}{[1pt]
  rgb(0pt)=(0.01,0.01,0.44);
  rgb(1pt)=(1,1,1);
  rgb(2pt)=(0.8, 0.25, 0.33)
},
point meta max=1,
point meta min=-1,
tick align=outside,
tick pos=left,
x grid style={darkgray176},
xmin=0.5, xmax=14.5,
xtick style={color=black},
y dir=reverse,
y grid style={darkgray176},
ymin=0.5, ymax=14.5,
ytick style={color=black},
xtick={1,2,4,6,8,10,12,14},
ytick={1,2,4,6,8,10,12,14},
]
\addplot graphics [includegraphics cmd=\pgfimage,xmin=0.5, xmax=14.5, ymin=14.5, ymax=0.5] {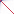};
\end{axis}

\end{tikzpicture}

    \vspace{-2pt}
    \caption{Correlation matrix for variables $\ell_1,\dots,\ell_{14}$ when $r_0\sim\mathrm{Beta}(5,50)$.}
    \label{fig:corrbeta}
\end{figure}
\end{minipage}%
\begin{minipage}{.04\linewidth}
\hfill
\end{minipage}%
\begin{minipage}[t]{.46\linewidth}
    \begin{figure}[H]
    \centering
    \begin{tikzpicture}[scale=0.8,xscale=.85]

\definecolor{darkgray176}{RGB}{176,176,176}

\begin{axis}[
very thick,
colorbar,
colorbar style={ylabel={}},
colormap={mymap}{[1pt]
  rgb(0pt)=(0.01,0.01,0.44);
  rgb(1pt)=(1,1,1);
  rgb(2pt)=(0.8, 0.25, 0.33)
},
point meta max=1,
point meta min=-1,
tick align=outside,
tick pos=left,
x grid style={darkgray176},
xmin=0.5, xmax=14.5,
xtick style={color=black},
y dir=reverse,
y grid style={darkgray176},
ymin=0.5, ymax=14.5,
ytick style={color=black},
xtick={1,2,4,6,8,10,12,14},
ytick={1,2,4,6,8,10,12,14},
]
\addplot graphics [includegraphics cmd=\pgfimage,xmin=0.5, xmax=14.5, ymin=14.5, ymax=0.5] {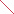};
\end{axis}

\end{tikzpicture}
    \vspace{-2pt}
    \caption{Correlation matrix for variables $\ell_1,\dots,\ell_{14}$ when $r_0\sim\mathcal{U}(0,1)$.}
    \label{fig:corrunif}
\end{figure}
\end{minipage}%
\begin{minipage}{.02\linewidth}
\hfill
\end{minipage}%
\end{figure}

\section{Conclusion}

In this paper, we studied convergence properties for a family of stochastic bisection methods in which the cut is chosen randomly according to any arbitrary distribution. When the position of the root is assumed to be uniformly distributed over the search interval, we derived explicit expected convergence rates depending only upon the mean and variance of the distribution of the cut. In particular, we showed that the relative length of the new working interval is independent from all previous information, which allowed us to compute the expected length of the interval at iteration $n$ in a manner similar to the classical deterministic bisection method. On average, the convergence is linear, and the expected interval length behaves as $C^n L_0$, where $C=1-2\,\mathbb{E}[c(1-c)]$ is the expected contraction factor.

We then extended our results to any initial absolutely continuous prior distribution of the root, provided that the cut distribution gives positive mass to the interior of the interval. Exploiting the stationarity of the uniform distribution for the normalized root process, we showed that the convergence rate approaches the one obtained in the uniform case. We also introduced and analyzed a multisection generalization in which several random cuts are chosen at each iteration, and we determined its average convergence behavior. The theoretical predictions were validated numerically through statistical simulations, which provided confidence intervals supporting the predicted convergence rates. These experiments also suggest that convergence toward the stationary regime may be slower when the prior distribution of the root is strongly asymmetric, while in symmetric situations the theoretical rates for the distribution of the root appear conservative and convergence occurs faster in practice.

Several questions remain open.

\begin{enumerate}
\item 
When the cut distribution is highly concentrated around the midpoint, numerical evidence suggests that the stochastic algorithm has a probability exceeding $50\%$ of converging at least as fast as the classical deterministic bisection method for a fixed root. This raises the problem of identifying optimal cut distributions that minimize the number of iterations required to achieve a prescribed tolerance with high probability.

\item 
In this work, the cut distribution was assumed to be fixed across iterations. A natural extension would allow the distribution to vary dynamically. While it is quite reasonable to assume that multiple results that we obtained would hold in this setting, one may wonder how the choice of the distributions influences the convergence rate of the algorithm, and if there exists an optimal sequence of distributions which maximizes the convergence rates for a given prior distribution of the root.

\item 
Our analysis was made in the special case where the function $f$ whose root we seek to estimate is unique in the interval $[a,b]$. One may wonder how allowing the possibility of $f$ having multiple roots in $[a,b]$ might affect our analysis. It is unclear to the present authors how this case would compare to the results of this paper. 

\item 
Finally, studying the algorithm within a more systematic framework of random dynamical systems and ergodic theory could provide a more modern interpretation of the results obtained here, and possibly give insight in the study of other similar dynamical systems.
\end{enumerate}

These findings illustrate how simple random perturbations of classical numerical procedures can generate rich probabilistic dynamics, potentially opening new directions for the analysis of stochastic algorithms and their convergence laws.

\bibliographystyle{plain}
\bibliography{ref}

\end{document}